\documentclass[11pt]{amsart}

\makeatletter

\def\chaptermark#1{}

\def\chapter{%
  \if@openright\cleardoublepage\else\clearpage\fi
  \thispagestyle{plain}\global\@topnum\z@
  \@afterindenttrue \secdef\@chapter\@schapter}

\def\@chapter[#1]#2{\refstepcounter{chapter}%
  \ifnum\c@secnumdepth<\z@ \let\@secnumber\@empty
  \else \let\@secnumber\thechapter \fi
  \typeout{\chaptername\space\@secnumber}%
  \def\@toclevel{0}%
  \ifx\chaptername\appendixname \@tocwriteb\tocappendix{chapter}{#2}%
  \else \@tocwriteb\tocchapter{chapter}{#2}\fi
  \chaptermark{#1}%
  \addtocontents{lof}{\protect\addvspace{10\p@}}%
  \addtocontents{lot}{\protect\addvspace{10\p@}}%
  \@makechapterhead{#2}\@afterheading}
\def\@schapter#1{\typeout{#1}%
  \let\@secnumber\@empty
  \def\@toclevel{0}%
  \ifx\chaptername\appendixname \@tocwriteb\tocappendix{chapter}{#1}%
  \else \@tocwriteb\tocchapter{chapter}{#1}\fi
  \chaptermark{#1}%
  \addtocontents{lof}{\protect\addvspace{10\p@}}%
  \addtocontents{lot}{\protect\addvspace{10\p@}}%
  \@makeschapterhead{#1}\@afterheading}
\newcommand\chaptername{Chapter}

\def\@makechapterhead#1{\global\topskip 7.5pc\relax
  \begingroup
  \fontsize{\@xivpt}{18}\bfseries\centering
    \ifnum\c@secnumdepth>\m@ne
      \leavevmode \hskip-\leftskip
      \rlap{\vbox to\z@{\vss
          \centerline{\normalsize\mdseries
              \uppercase\@xp{\chaptername}\enspace\thechapter}
          \vskip 3pc}}\hskip\leftskip\fi
     #1\par \endgroup
  \skip@34\p@ \advance\skip@-\normalbaselineskip
  \vskip\skip@ }
\def\@makeschapterhead#1{\global\topskip 7.5pc\relax
  \begingroup
  \fontsize{\@xivpt}{18}\bfseries\centering
  #1\par \endgroup
  \skip@34\p@ \advance\skip@-\normalbaselineskip
  \vskip\skip@ }
\def\appendix{\par
  \c@chapter\z@ \c@section\z@
  \let\chaptername\appendixname
  \def\thechapter{\@Alph\c@chapter}}

\newcounter{chapter}

\newif\if@openright

\makeatother

\renewcommand\thechapter{\Alph{chapter}}

\usepackage{amsmath,amssymb,amsthm}
\usepackage{enumerate}
\usepackage{mathtools}
\usepackage{tikz}
\usepackage{graphicx}
\usepackage{tikz-cd}
\usepackage[textsize=small]{todonotes}
\usepackage{xparse}
\usepackage{xspace}
\usepackage{mathtools}
\usepackage[all]{xy}
\usepackage{xcolor}
\usepackage{verbatim}
\usepackage[colorlinks,urlcolor=blue,linkcolor=blue,citecolor=blue]{hyperref}
\usepackage{mathrsfs}
\usepackage{microtype}
\usepackage{a4wide}
\usepackage{cleveref}
\usepackage{xspace}
\usepackage{stmaryrd}

\usepackage{scalerel}
\usepackage{stackengine}

\usepackage{longtable} 
\usepackage{pdflscape} 
\usepackage{array}

\numberwithin{equation}{section}

\newcommand{\N}{{\mathbb N}}
\newcommand{\Z}{{\mathbb Z}}

\newcommand{\Q}{{\mathbb Q}}
\newcommand{\R}{{\mathbb R}}
\newcommand{\C}{{\mathbb C}}
\newcommand{\U}{{\mathcal U}}

\newcommand{\GL}{\operatorname{GL}}
\newcommand{\PSL}{\operatorname{PSL}}
\newcommand{\Ma}{\operatorname{M}}

\newcommand{\I}{\mathrm{I}}
\newcommand{\V}{\mathrm{V}}
\newcommand{\ZZ}{\mathcal{Z}}
\newcommand{\Cliff}{{\mathcal{C}}}  

\newcommand{\B}{\operatorname{B}\xspace}

\definecolor{wildstrawberry}{rgb}{1.0, 0.26, 0.64}

\renewcommand{\O}{\mathcal{O}}
\newcommand{\OO}{\mathcal{O}}
\newcommand{\LL}{\mathcal{L}}
\newcommand{\E}{\operatorname{E}}

\newcommand{\GE}{\operatorname{GE}}
\newcommand{\D}{\operatorname{D}}
\newcommand{\DE}{\operatorname{DE}}
\newcommand{\sm}[1]{\left(\begin{smallmatrix} #1 \end{smallmatrix}\right)}
\newcommand{\SL}{\operatorname{SL}}

\newcommand{\qa}[3]{\left(\frac{#1, #2}{#3}\right)}
\newcommand{\FA}[1][]{\ifx #1\textup{FA}\xspace \else
  \textup{FA$_{#1}$}\xspace
  \fi
}
\newcommand{\HFA}[1][]{\ifx #1\textup{HFA}\xspace \else
  \textup{HFA$_{#1}$}\xspace
  \fi
}

\newcommand{\T}{\textup{(T)}\xspace}

\newtheorem{lemma}{Lemma}[section]
\newtheorem{proposition}[lemma]{Proposition}
\newtheorem{theorem}[lemma]{Theorem}
\newtheorem{corollary}[lemma]{Corollary}

\newtheorem{maintheorem}{Theorem}

\theoremstyle{definition}

\newtheorem{definition}[lemma]{Definition}

\newtheorem{question}[lemma]{Question}
\newtheorem{remark}[lemma]{Remark}
\newtheorem*{remark*}{Remark}
\newtheorem*{proposition*}{Proposition}
\newtheorem*{problem*}{Problem}

\title[A dichotomy for integral group rings]{A dichotomy for integral group rings via higher modular groups as amalgamated products}
\author[A.~B\"achle]{Andreas B\"achle}
\author[G.~Janssens]{Geoffrey Janssens}
\author[E.~Jespers]{Eric Jespers}
\author[A.~Kiefer]{Ann Kiefer}
\author[D.~Temmerman]{Doryan Temmerman}

\address{Vakgroep Wiskunde, Vrije Universiteit Brussel, Pleinlaan 2, 1050 Brussels, Belgium}
\email{\href{mailto:Andreas.Bachle@vub.be}{Andreas.Bachle@vub.be}, \href{mailto:Geoffrey.Janssens@vub.be}{Geoffrey.Janssens@vub.be}, \href{mailto:Eric.Jespers@vub.be}{Eric.Jespers@vub.be}, \href{mailto:Ann.Kiefer@vub.be}{Ann.Kiefer@vub.be}, \href{mailto:Doryan.Temmerman@vub.be}{Doryan.Temmerman@vub.be}}

\thanks{The first and second author are grateful to Fonds Wetenschappelijk Onderzoek - Vlaanderen for financial support. The third, fourth and fifth author are grateful to Onderzoeksraad VUB and Fonds Wetenschappelijk Onderzoek - Vlaanderen for financial support.}

\begin{document}

\begin{abstract}
We show that $\mathcal{U}(\mathbb{Z}G)$, the unit group of the integral group ring $\mathbb{Z} G$, either satisfies Kazhdan's property (T) or is, up to commensurability, a non-trivial amalgamated product, in case $G$ is a finite group satisfying some mild conditions. 
Crucial in the proof is the construction of amalgamated decompositions of the elementary group $\operatorname{E}_2(\mathcal{O})$, where $\mathcal{O}$ is an order in a rational division algebra. A major step is to introduce subgroups $\operatorname{E}_2(\Gamma_n(\mathbb{Z}))$ inside the so-called higher modular groups $\operatorname{SL}_+(\Gamma_n(\mathbb{Z}))$, which are discrete subgroups of certain $2 \times 2$ matrix groups with entries in a Clifford algebra. The groups $\operatorname{E}_2(\Gamma_n(\mathbb{Z}))$ mimic the elementary groups in linear groups over rings. We prove that $\operatorname{E}_2(\Gamma_n(\mathbb{Z}))$ has in general a non-trivial decomposition as a free product with amalgamated subgroup $\operatorname{E}_2(\Gamma_{n-1}(\mathbb{Z}))$. From this we obtain that also the higher modular groups do have a very clearly structured amalgam decompositions in low dimensions.
\end{abstract}

\maketitle

\section{Introduction}

The paper has two principal aims. The first is representation theoretical and concerns the program initiated in \cite{abelianizationpaper} to disentangle $\mathcal{U}(\Z G)$ using geometric group theoretical notions, and this without losing track of the torsion elements.
The ultimate goal hereof would be to develop reduction methods for certain guiding questions in integral representation theory such as the integral isomorphism problem and the Zassenhaus conjectures.
The main contribution of this paper to this goal is \Cref{dichotomy HFA-Amalgamated} which says that, under a mild condition on the finite group $G$, the unit group $\mathcal{U}(\Z G)$ either has Kazhdan's property $\T$ or is, up to commensurability, a non-trivial amalgamated product $ A \ast_C B$ (and hence in this case all torsion subgroups are conjugated to subgroups in $A$ or $B$).
In order to obtain (useful) reduction methods one needs in the second case to produce concrete amalgamations, whereas in the first case an abundance of literature equips us with new machinery.
Interestingly, combining the recently obtained characterisation of property \T of the authors \cite[Theorem 7.1]{abelianizationpaper} with Jespers-Leal's theorem \cite[Theorem 3.3]{JesLea2} one obtains that property \T occurs (almost) exactly when the generic constructions of bicyclic and Bass units \cite[Chapter 1]{EricAngel1} form a subgroup of finite index in $\U (\Z G)$.
Consequently, one may confidently expect that the dichotomy mentioned above will form a pillar for future research in the above program. 

The proof of the dichotomy ultimately relies on understanding the structure of the linear groups $\SL_2(\mathcal{O})$ and their subgroups generated by elementary matrices $\E_2(\mathcal{O})$ with $\mathcal{O}$ an order in a rational division algebra $D$.
In particular, it will rely on the modular group $\PSL_2(\mathbb{Z})$ and the Bianchi groups $\PSL_2(\mathcal{I}_d)$, where $\mathcal{I}_d$ is the ring of integers of $\Q(\sqrt{-d})$.
Bianchi groups have already a long history and presentations and amalgamated decompositions (when they exist) are well known, for example see the books \cite{FineBook, Serre}.
However, to obtain the dichotomy we cannot escape from also considering special linear groups over smaller orders in quadratic imaginary extensions of $\Q$. Moreover, we are forced to handle the by far less well understood cases when $\O$ is an order in a quaternion algebra.

One of the first papers handling units in $2 \times 2$ matrix groups over orders in quaternion algebras is \cite{clifford}. It is shown that it all comes down to studying discrete subgroups of Vahlen groups, which are $2 \times 2$ matrices with entries in a Clifford algebra. These groups were described in \cite{Vahlen}, where Vahlen generalized M\"obius transformations to higher dimensional hyperbolic spaces. Discrete subgroups of the latter provide a natural generalization of the modular group to higher dimensions. These generalized groups are denoted by $\SL_+(\Gamma_n(\Z))$ and referred to as higher dimensional modular groups. These groups were rediscovered later by Ahlfors in \cite{1985ahlfors}.  

In this paper, we consider subgroups of $\SL_+(\Gamma_n(\Z))$ generated by elementary matrices, which we denote by $\E_2(\Gamma_n(\Z))$. 
We will see that, although $\Gamma_n(\Z)$ is not a ring, the newly defined group $\E_2(\Gamma_n(\Z))$ mimics remarkably well the classical group $\E_2(R)$, for a ring $R$, and its use for the study of $\GL_2(R)$.
Next, in the spirit of \cite{Cohn1, FineBook}, a constructive method is given to obtain useful finite presentations for the elementary subgroup.
Using the latter presentations, we then obtain concrete non-trivial amalgamated decompositions of $\E_2(\Gamma_n(\Z)$ for $n \geq 1$. 

The modular and Bianchi groups have numerous applications throughout number theory, geometry and analysis.
One can be optimistic that the higher-dimensional analogue of these groups will turn out to have similar uses (for instance see \cite{McInroy}) and hence these groups constitute a topic of independent interest. This is why, although the important case in the group ring context is the case when $n=4$, we consider the groups $\E_2(\Gamma_n(\Z))$ in their whole generality for all $n \geq 1$.

We will now highlight in two parts the main results obtained in this paper. 

\subsection*{Presentations and amalgamations of higher modular groups}
Vahlen \cite{Vahlen} introduced a special linear group of $2\times 2$ matrices with entries in the Clifford group $\Gamma_n(\R)$ of the real Clifford algebra $\mathcal{C}_n(\mathbb{R})$ of degree $n$. In \cite{Maass} Maass also described this group. Recall that $\mathcal{C}_n(\mathbb{R})$ is the (associative and unital) $\mathbb{R}$-algebra generated by $n-1$ symbols $i_1, \ldots, i_{n-1}$ and solely subject to the relations $i_k^2 = -1$ and $i_t i_k = - i_k i_t$ for $t\neq k$. The Clifford group $\Gamma_n(\R)$ is the group of invertible elements in $\mathcal{C}_n(\R)$.
They prove that this special linear group acts via M\"obius transformations on the hyperbolic space $\mathbb{H}^{n+1}$ of dimension $n + 1$.

In 1985, in \cite{1985ahlfors}, Ahlfors draws attention to these special linear groups, mentioned in the paragraph before,  and denotes these groups by $\SL_+(\Gamma_n(\R))$. In this paper we consider their discrete subgroups $\SL_+(\Gamma_n(\Z))$, which act discontinuously on $\mathbb{H}^{n+1}$. This group generalizes the classical case of Fuchsian and Kleinian groups, that act discontinuously on hyperbolic space of dimension $2$ and $3$ respectively, to higher dimensions.
Recently generalizations of Kleinian groups to higher dimensions \cite{Kapovich} or Vahlen groups in general \cite{McInroy} have received more attention.

In \cite{Cohn1} Cohn initiated, for general rings $R$, the systematic study of the group $\E_2(R)$ generated by the elementary matrices over $R$. Nowadays, their predominant role in the study of $\GL_2(R)$ goes without saying.
For example, building on Cohn's work, Fine \cite{FineBook} was able, among other things, to obtain amalgam decompositions for the Euclidian Bianchi groups.

With this in mind, we initiate in \Cref{introduction clifford} the study of the group generated by elementary matrices in $\SL_+(\Gamma_n(\Z))$, denoted $\E_2(\Gamma_n(\Z))$, and investigate this subgroup in the subsequent sections.
As a first main theorem we show that, to our surprise, the elementary groups are nested into each other in a very precise structural way.
As a consequence we obtain an amalgam decomposition for the ``Euclidean'' higher modular groups, i.e. in case $n \leq 4$.

\begin{maintheorem}[Theorem~\ref{prop:E_2_O_L_is_amalgam} and Corollary~\ref{SL_2L_2isamalgam}]\label{mainTheointro}
The group $\E_2(\Gamma_n(\Z))$ has a non-trivial decomposition as an amalgamated product with amalgamated subgroup $\E_2(\Gamma_{n-1}(\Z))$. In particular, for $n \leq 4$, $\SL_+(\Gamma_n(\Z))$ is an amalgamated product over $\SL_+(\Gamma_{n-1}(\Z))$.
\end{maintheorem}

To prove the result, we first carefully explain in \Cref{finite presentation clifford} how to obtain a finite presentation of $\E_2(\Gamma_n(\Z))$ for all $n \geq 1$.
Thereafter, inspired by methods in \cite{FineBook}, we obtain in \Cref{mainsection} a surprisingly smooth concrete amalgamated decomposition, yielding the main theorem above.

\begin{remark*}
For $n=1$ and $n=2$ the higher modular group $\SL_+(\Gamma_n(\Z))$ is isomorphic to $\SL_2(\Z)$ and $\SL_2(\Z[i])$, respectively. For $n=3$, we get an isomorphism with the group described in \cite{MacWatWie} and for $n=4$, this group is isomorphic to a subgroup of finite index in the group of $2 \times 2$ matrices of reduced norm $1$ over the Lipschitz quaternions $\LL$. Hence, \Cref{mainTheointro} generalizes the classical results for $\SL_2(\Z)$ and $\SL_2(\Z[i])$, but it also yields a previously unknown amalgamation of $\SL_+(\Gamma_3(\Z))$ over $\SL_2(\Z[i])$ and $\SL_+(\Gamma_4(\Z))$ over  $\SL_+(\Gamma_3(\Z))$.
\end{remark*}
 
\subsection*{A dichotomy for integral group rings}

Let $G$ be a finite group. One of the most natural questions in representation theory is the so-called integral isomorphism problem which asks whether $G$ is uniquely determined by the integral group ring $\Z G$. This question was already posed explicitly for the first time by Higman in 1940 \cite{Higman} and thereafter popularized by Brauer \cite{BrauerSurvey} in the 1960's.

Interestingly, this problem has an equivalent purely group theoretical reformulation, namely whether the unit group $\U (\Z G)$ determines $G$ uniquely. This is one of the reasons why the structure of $\U (\Z G)$ received tremendous attention in the last five decades and we refer to the books \cite{EricAngel1, EricAngel2, SehgalBook93} for the main advances on this topic. Over time two main research directions emerged. On the one hand, the search for generic constructions of subgroups of finite index in $\U (\Z G)$ and on the other hand the understanding of torsion units in $\Z G$ with a special emphasis on the Zassenhaus conjectures. The third and strongest of these conjectures, denoted (ZC3), asserts that every finite subgroup of $\U (\Z G)$ is conjugated, within $\U (\Q G)$, to a subgroup of $G$. Since the groundbreaking work of Roggenkamp-Scott \cite{RoggScott} and Weiss \cite{WeissPGroup, WeissNilpotent}, proving (ZC3) for nilpotent groups, progress on (ZC3) and the isomorphism problem has been scarce. The main issues for further progress on the latter questions are the following

\begin{enumerate}
\item The lack of (generic) constructions of (torsion) units. 
\item The absence of reduction methods.
\end{enumerate}

As mentioned by Kleinert in \cite{KleinertSurvey} a ``Unit Theorem" for $\U(\Z G)$ (i.e. a basic structure theorem such as Dirichlet's unit theorem)  is still missing.
For abelian groups this was solved by Higman \cite{Higman}, however for only very few non-abelian groups $G$ a full presentation of $\U (\Z G)$ is known. A main reason for this is the first problem mentioned above. According to Kleinert \cite{KleinertSurvey}, ``a unit theorem for a finite-dimensional (semi-)simple rational algebra $A$ consists of the definition, in purely group theoretical terms, of a class of groups $\mathcal{C}(A)$ such that almost all generic unit groups of $A$ are members of $\mathcal{C}(A)$.''
Here a generic unit group stands for a torsion-free subgroup of finite index in $\SL_1 (\O)$, the elements of reduced norm $1$ of some order $\O$ in $A$.
The reason one considers $\SL_1(\O)$ instead of $\U (\O)$ is that $\langle \SL_1(\O), \mathcal{Z}(\U(\O)) \rangle$ is always a subgroup of finite index in $\U (\O)$. By Dirichlet's unit theorem, the structure of $\mathcal{Z}(\U (\O))$ is then considered as known.
In other words, the best one can hope for is to describe in a generic way representatives of the commensurability class of $\U (\Z G)$.
This point of view was explicitly formulated in \cite{JesRioCrelle} and called the virtual structure problem:

\begin{problem*}[The virtual structure problem for unit groups of integral group rings] For a class of groups $\mathcal{G}$, classify the finite groups $G$, such that $\U (\Z G)$ is virtually in $\mathcal{G}$.
\end{problem*}

In \cite{JPdRRZ}, the authors solved the problem in case $\mathcal{G}$ is the class consisting of direct products of free-by-free groups. Joint with the results in \cite{abelianizationpaper} we obtain (almost) the virtual structure problem for the class consisting of the non-trivial amalgamated products.

\begin{maintheorem}[Theorem~\ref{dichotomy HFA-Amalgamated}]\label{mainTheoBintro}
Let $G$ be finite group which is solvable or $5 \nmid |G|$. If $\U(\Z G)$ has finite center, then exactly one of the following properties holds:
\begin{enumerate}
\item $\mathcal{U}(\mathbb{Z}G)$ has property \T.
\item\label{it:com_amprod} $\mathcal{U}(\mathbb{Z}G)$ is commensurable with a non-trivial amalgamated product.
\end{enumerate}
\end{maintheorem}

Recall that by the Delorme-Guichardet Theorem \cite[Theorem 2.12.4]{BdlHV} a countable discrete group $\Gamma$ has property \T if and only if every affine isometric action of $\Gamma$ on a real Hilbert space has a fixed point. \Cref{mainTheoBintro} lays the first pillar towards the solution of the second problem, i.e. ``the absence of reduction methods''. Indeed,  in case $\U(\Z G)$ is a non-trivial free product with amalgamation  $A \ast_C B$, the study of torsion units, is reduced to (conjugates) of the proper subgroups $A$ and $B$ of $\U(\Z G)$. The next milestone would be a generic construction of the amalgamated product mentioned in the theorem. Recently in \cite{JanJesTem}, building on \cite{GonPas}, the first generic constructions of free products of torsion subgroups in $\U (\Z G)$ have been found.

In \cite[Theorem 7.1 and Corollary 7.5]{abelianizationpaper} the authors obtained the following characterization, in terms of $G$, of when property \T occurs for $\U(\Z G)$. It is proven that the group $\U(\Z G)$ has property \T if and only if $\U(\Z G)$ has finite center and the group algebra $\Q G$ does not have $\Ma_2(D)$ as an epimorphic image, where $D$ is
\begin{equation}\label{exc components}
\Q, \Q(\sqrt{-d}), \qa{a}{b}{\Q}
\end{equation}
 with $d\geq 0$ and $a, b <0$. Therefore, in order to prove \Cref{mainTheoBintro}, we need to understand $\SL_2(\O)$, with $\O$ an order in any of these algebras. This will be done through its elementary subgroup $\E_2(\O)$, and here the amalgamated products obtained in the first part of the paper play an important role. We, furthermore, obtain the following result which we believe is interesting on its own.

Put $\B'_2(R) = \E_2(R) \cap \B_2(R)$, for any ring $R$, where $\B_2(R)$ denotes the Borel subgroup of $\GL_2(R)$ (i.e. the group of invertible upper triangular $2\times 2$-matrices).

\begin{maintheorem}[\Cref{Zsqrt3amalgam}]
Let $\mathcal{O}$ be an order in $\mathbb{Q}(\sqrt{-d})$ with $d \geq 0$ a square-free integer. Then  $\E_2(\mathcal{O})$ has a decomposition as a non-trivial amalgamated product if and only if $\O \neq \mathcal{I}_3$. 
Moreover, if $\mathcal{O}$ is different from $\mathbb{Z}[\sqrt{-3}]$ and also from $\mathcal{I}_d$ with $d \in \lbrace 1,2,3,7,11 \rbrace$, then 
\begin{equation*}
\E_2(\mathcal{O}) \cong \E_2(\mathbb{Z}) \ast_{\B'_2(\mathbb{Z})} \B'_2(\O).
\end{equation*}
\end{maintheorem}

Note that the rings $\mathbb{Z}[\sqrt{-3}]$ and $\mathcal{I}_d$, with $d = 1,2,3,7,11$, are exactly those orders in imaginary quadratic extensions of $\Q$ which are $\GE_2$-rings, i.e. the general linear group is generated by elementary and diagonal matrices (see \cite[Theorem 3]{Dennis}). 
Interestingly, for $d \not \in \lbrace 1,2,3,7,11 \rbrace$, Frohman and Fine \cite{FroFin} proved that $$\SL_2(\mathcal{I}_d) \cong \E_2(\mathcal{I}_d) \ast_{F_d} G_d$$
for certain concrete groups $F_d$ and $G_d$. Hence, our description complements their result.
Note that for $\O=\mathcal{I}_3$, it is well-known that the Bianchi group $\PSL_2(\mathcal{I}_3)$ satisfies Serre's property \FA. In particular, $\SL_2(\mathcal{I}_3)$ has no non-trivial decomposition as an amalgamated product. The group $\SL_2(\mathbb{Z}[\sqrt{-3}])$ is a subgroup of index $2$ in $\SL_2(\mathcal{I}_3)$ and it comes a bit as a surprise that it does have a non-trivial amalgamated decomposition.

To end, in case $\U (\Z G)$ does not have property \T, we investigate which concrete subgroup is an amalgamated product. By \cite[Appendix A]{abelianizationpaper}, when $\Q G$ has an epimorphic image of the form $\Ma_2(D)$, with $D$ as listed in (\ref{exc components}), it will also have $\Ma_2(\Q)$ as an epimorphic image. This condition is in fact equivalent for $G$ to have $D_8$ or $S_3$ as an epimorphic image, see \Cref{remark over wnr M-2(Q) component}. One of the results obtained in this context is the following.

\begin{maintheorem}[\Cref{amalgam voor QG als M_2(Q) component}]
Let $G$ be a finite group having $D_8$ or $S_3$ as an epimorphic image. Then $\U(\Z G)$ is virtually a non-trivial amalgamated product. Furthermore, every subgroup $H$ of finite index in $\SL_1(\Z G)$ has a non-trivial amalgamated decomposition. 
\end{maintheorem}

The outline of the paper is as follows. In Section~\ref{preliminaries}, we give the necessary background on free products with amalgamation, $2$-by-$2$ matrices over quaternion algebras and higher modular groups. In Section~\ref{introduction clifford} we initiate the $\E_2$ and $\GE_2$ part of the higher modular groups. Thereafter, in \Cref{finite presentation clifford}, we give a finite concrete presentation of the latter groups.  In Section~\ref{mainsection} we succeed in obtaining a non-trivial amalgamated decomposition for $\E_2( \Gamma_n(\Z )$ over $\E_2(\Gamma_{n-1}(\Z))$. Finally in the last section \ref{sectiongroupring} we consider $\U (\Z G)$, prove the dichotomy and obtain the refinements obtained above. In the appendix we give a non-exotic example of a quaternion order which is not universal for $\GE_2$.

\section{Preliminaries}\label{preliminaries}

\subsection{Free products with amalgamation}

First we recall the definition of a free product with amalgamation.

\begin{definition}
Let $G_1$, $G_2$ and $H$ be groups and $f_1: H \rightarrow G_1$ and $f_2: H \rightarrow G_2$ be  injective homomorphisms. Let $N$ be the normal subgroup of the free product $G_1 \ast G_2$ generated by elements of the form $f_1(h)f_2(h)^{-1}$ for $h \in H$. Then the \textit{free product with amalgamation} $G_1 \ast_H G_2$ is defined as the quotient $(G_1 \ast G_2) / N$. This free product with amalgamation is said to be trivial if $f_1$ or $f_2$ are surjective. If not stated differently, we always assume the free product with amalgamation to be non-trivial. The group $H$ is called the \emph{amalgamated subgroup}. 
\end{definition}

Throughout the paper, abusing terminology, we will often abbreviate the term ``free product with amalgamation'' simply as ``amalgamated product''. 

The next proposition is standard (see for example \cite[Lemma 3.2]{Zie}).

\begin{proposition}\label{amalgamconserved}\label{amalgamdirectproduct}
Let $G=G_1 \ast_H G_2$ be a free product with amalgamation. If $\Phi$ is an epimorphism from $\tilde{G}$ to $G$, then $\tilde{G}$ is the free product with amalgamation $\Phi^{-1}(G_1) \ast_{\Phi^{-1}(H)} \Phi^{-1}(G_2)$. In particular, for any group $H$, the direct product $G \times H$ is again a free product with amalgamation. 
\end{proposition}

The following proposition is folklore, but as we will use it later on in the paper, we state it here explicitly. 

\begin{proposition}\label{disjoint generators amalgam}
Let $G$ be a group and let $A,B,C \subseteq G$ be pairwise disjoint subsets such that $$G:= \langle A\cup B \cup C \mid R_{AC}, R_{BC}\rangle,$$ where $R_{AC}$, $R_{BC}$ is a collection of relations on the generators in $A\cup C$ and $B\cup C$ respectively. Write $G_{AC}$, $G_{BC}$ and $G_C$ for the subgroup of $G$ generated by $A\cup C$, $B \cup C$ and $C$ respectively. Then $$G \cong G_{AC} \ast_{G_C} G_{BC}.$$
\end{proposition}
As $G_C \subseteq G_{AC} \cap G_{BC}$, the injections of $G_C$ in both components in the amalgamation are the canonical injections.

\subsection{A matrix ring over a totally definite rational quaternion algebra}\label{matquat}

Recall that the rational quaternion algebra $\left( \frac{u,v}{\Q}\right)$ is the $\Q$-vector space with basis $\{1,i ,j, k \}$ and multiplication determined by the relations
$$i^2 = u,\quad j^2 = v,\quad ij = k = -ji,$$
where $u,v \in \Q \setminus \{0\}$. This quaternion algebra is said to be totally definite if $u$ and $v$ are negative.

Let $a = a_1\cdot 1 + a_2\cdot i + a_3\cdot j + a_4\cdot k \in \left( \frac{u,v}{\Q}\right)$, with $a_1, a_2, a_3, a_4 \in \Q$. The \emph{real part} of $a$ is defined to be $\operatorname{Re}(a) = a_1$ and we put $\overline{a}=a_1\cdot 1 - a_2\cdot i - a_3\cdot j - a_4\cdot k$ and $|a| = \sqrt{a\overline{a}}$. It is well-known that the latter defines an analytic norm (i.e.\ it also  satisfies the triangle inequality) on $\left( \frac{u,v}{\Q}\right)$ if and only if the quaternion algebra is totally definite. The square, i.e.\ the map $a \mapsto a\overline{a}$, makes some quaternion algebras right Euclidean, a notion which generalizes the notion of Euclidean rings to the non-commutative setting.  
\begin{definition}\label{Euclidean}
Let $R$ be a ring and $\delta: R \rightarrow \N$ a map with $\delta(0) = 0$.
We call $R$ a \emph{right Euclidean} ring if $$\forall\ a,b \in R \text{ with } b\neq 0, \exists\  q,r \in R: a= bq + r\text{ with } \delta(r) < \delta(b) \text{ or } r = 0.$$
\end{definition}

By \cite{Fitz} the only totally definite rational quaternion algebras having a right Euclidean order are $\qa{-1}{-1}{\Q}$,$\qa{-1}{-3}{\Q}$ and $\qa{-2}{-5}{\Q}$, where $\delta$ is the square map. Recall that a \emph{$\Z$-order} (or for brevity just \emph{order}) of a finite-dimensional rational algebra $A$ over $\Q$ is a subring of $A$ that is finitely generated as a $\Z$-module and contains a $\Q$-basis of $A$. Moreover in the mentioned cases the right Euclidean order is the (up to conjugation) unique maximal order.

For our purposes, we need to deal with a generalization of the special linear group over a totally definite rational quaternion algebra. To do so, we consider the elements of reduced norm $1$ in $2$-by-$2$ matrix rings over an order in a quaternion algebra $\qa{u}{v}{\Q}$. 

For convenience, we recall some definitions and standard notations. Let $A$ be a finite dimensional central simple algebra over a field $K$. Let $E$ be a splitting field of $A$, i.e. $E \otimes_K A \cong \Ma_{n}(E)$, a full matrix ring over $E$. The \emph{reduced norm} of $a \in A$ is defined as \[\operatorname{RNr}_{A/K}(a) = \det(1 \otimes a). \] 
Note that $\operatorname{RNr}_{A/K}$ is a multiplicative map, $\operatorname{RNr}_{A/K}(A) \subseteq K$ and $\operatorname{RNr}_{A/K}(a)$ does only depend on $K$ and $a \in A$ (and not on the chosen splitting field $E$ and isomorphism $E \otimes_K A \cong \Ma_{n}(E)$), see \cite[page 51]{EricAngel1}. For a subring $R$ of $A$, put
 \[ \SL_1(R) = \{ a \in \U(R)\ |\ \operatorname{RNr}_{A/K}(a) = 1 \}, \]
 which is a (multiplicative) group. If $A = \Ma_n(D)$ and $R= \Ma_n(\O)$ with $\O$ an order in $D$, then we also write $\SL_1(A)= \SL_n(D)$ and $\SL_1(R) = \SL_n(\O)$.

\begin{remark}\label{Dieudonneremark} For $\OO$ an order in $\qa{u}{v}{\Q}$ the group $\SL_2(\OO)$ may alternatively be defined via the so called Dieudonn\'e determinant. For more details on Dieudonn\'e determinants we refer the interested reader to the very accessible paper \cite{Aslaksen} or to \cite{Cohn3, draxl}. The following formula for the \emph{Dieudonn\'e determinant} $\Delta$ of $\sm{a & b \\ c & d} \in  \Ma_2\left(\qa{u}{v}{\Q}\right)$ can be derived.
\begin{equation}\label{defDieu}
 \Delta \left(\sm{a & b \\ c & d }\right) = \sqrt{\vert a \vert^2\vert d \vert^2 + \vert b \vert^2\vert d \vert^2 - 2\mathrm{Re}\left[a\overline{c}d\overline{b}\right]}.
\end{equation}
There is a link between the Dieudonn\'e determinant of a matrix and the reduced norm. Indeed, by \cite[Theorem~1, page~146]{draxl}, for $A=\sm{a & b \\ c & d} \in \Ma_2\left(\qa{u}{v}{\Q}\right)$, we get that 
\begin{equation}\label{deltared}
\Delta(A)^2 = \mbox{RNr}_{\Ma_2\left(\qa{u}{v}{\Q}\right)/\Q}(A)  .
\end{equation}
This shows that the reduced norm of a matrix in $\Ma_2\left(\qa{u}{v}{\Q}\right)$ is rational and positive and thus the groups $\GL_2\left(\O\right)$ and $\SL_2\left(\O\right)$ are equal.
\end{remark}

\subsection{Clifford algebras}
We recall the concepts of Clifford algebras and Clifford groups. This is based on \cite{1985ahlfors} and \cite{ElsGrunMen}.
Let $R$ be a unital associative ring. The \emph{Clifford algebra $\Cliff_{n}(R)$} (with $n \geq 1$) is the unital associative $R$-algebra generated by $n-1$ elements $i_1, \ldots, i_{n-1}$ subject to the following relations 
\begin{equation}
i_hi_k = -i_ki_h, \textrm{ for } h \neq k\quad \textrm{ and }\quad i_h^2=-1 \textrm{ for } 1 \leq h \leq n-1.
\end{equation}
Let $\V^n(R)$ denote the free $R$-submodule of $\Cliff_n(R)$ with basis $\lbrace i_0 := 1, i_1, \ldots, i_{n-1} \rbrace$.
The image of $a \in \Cliff_n(R)$ under the linear automorphism determined by $i_h \mapsto -i_h$ is denoted  by $a'$. The map $a \mapsto a^{*}$ consists in reversing the order of the factors in a basis element $i_{h_1}\ldots i_{h_m}$ of the algebra. Note that $(a b)^*=b^*a^*$.
Finally the composition of these two maps gives another map denoted by $a \mapsto \overline{a}=a'^{*}=\left(a^{*}\right)'$. Note that if $a \in \V^n(R)$, then $a^* = a $ and $a' = \overline{a}$. 

Consider now $R=\R$. For small values of $n$, one has the following isomorphisms.
$$\Cliff_1(\R) \cong \R,\ \Cliff_2(\R) \cong \C \textrm{ and } \Cliff_3(\R) \cong \qa{-1}{-1}{\R}.$$
The vector space $\Cliff_{n}(\R)$ of $\R$-dimension $2^{n-1}$ is endowed with the Euclidean norm, \[|a| = \left(\sum a_I^2 \right)^\frac{1}{2} \quad \text{ for } a = \sum a_I I \in \Cliff_n(\R),\  I=i_{j_1} \cdots i_{j_r},\ 1 \leq j_1 < ... < j_r \leq n-1.
\]
It is clear that if $a \in \V^n(\R)$, then $a\overline{a}=\vert a \vert^2$. Thus the inverse of every non-zero $a \in \V^n(\R)$ is given by $a^{-1}=\overline{a}\vert a \vert^{-2}$.
The real \emph{Clifford group $\Gamma_n(\R)$} is the multiplicative group of all products of invertible vectors of $\V^n(\R)$. By \cite[Lemma 1.2]{1985ahlfors}, for $a,b \in \Gamma_n(\R)$, $\vert ab \vert=\vert a \vert\vert b \vert$.

This allows to define the general linear group and the special linear group over the Clifford group.

\begin{definition}[{\cite[Definition 2.1]{1985ahlfors}}]\label{defcliffmoeb}
\begin{eqnarray*}
\GL(\Gamma_n(\R)) & = & \left \lbrace \begin{pmatrix} a & b \\ c & d \end{pmatrix} \mid a,b,c,d \in \Gamma_n(\R) \cup \lbrace 0 \rbrace,\ ad^*-bc^* \in \R\setminus \lbrace 0\rbrace, \right.\\
& & \phantom{ \left \lbrace \begin{pmatrix} a & b \\ c & d \end{pmatrix} \mid \right. } \left. ab^*,\ cd^*,\ c^*a,\ d^*b \in \V^n(\R) \right \rbrace,\\
\SL_+(\Gamma_n(\R)) & = & \left \lbrace \begin{pmatrix} a & b \\ c & d \end{pmatrix} \in \GL(\Gamma_n(\R))  \mid ad^*-bc^* =  1 \right \rbrace.
\end{eqnarray*}
\end{definition}

One can straightforwardly prove that the map $\GL(\Gamma_n(\R)) \to \R\setminus \lbrace 0\rbrace \colon \sm{a & b \\ c & d} \mapsto ad^* - bc^*$ is multiplicative. Moreover, all the above definitions make sense when replacing the field $\R$ by $\Q$. Hence we also get the rational Clifford group $\Gamma_n(\Q)$ and $\SL_+(\Gamma_n(\Q))$.

In \cite{1985ahlfors}, it has been shown that $\GL(\Gamma_n(\R))$ and $\SL_+(\Gamma_n(\R))$ are multiplicative groups. The same is true when replacing $\R$ by $\Q$. For full disclosure, in \cite[Definition 3.1]{ElsGrunMen} a significantly more involved definition of $\SL_+(\Gamma_n(\R))$ is given, but it has been later shown in \cite[Theorem 3.7]{ElsGrunMen} that both definitions are equivalent. 

Although the real Clifford algebra of dimension $3$ is isomorphic to the classical quaternion algebra over $\R$, this isomorphism does not extend to an isomorphism between $\SL_2\left(\qa{-1}{-1}{\R}\right)$, as defined in \Cref{matquat}, and $\SL_+(\Gamma_3(\R))$. However, as the following theorem shows, the special linear group $\SL_2\left(\qa{-1}{-1}{\R}\right)$ is isomorphic to the special linear group over a higher Clifford group.

\begin{theorem}[{\cite[Section 6]{ElsGrunMen}}]\label{isocliffquat}
The groups $\SL_2\left(\qa{-1}{-1}{\R}\right)$ and $\SL_+(\Gamma_4(\R))$ are isomorphic.
\end{theorem}

The following lemma is well-known and will be useful later in concrete computations. 

\begin{lemma}\label{simulVn}
Let $R$ be equal to $\R$ or $\Q$. Let $a, b, c, d \in \Gamma_n(R)$ and $x \in \V^n(R)$. The following properties hold.
\begin{enumerate}
\item Either both or none of $ab^{-1}$ and $a^*b$ belong to $\V^n(R)$. Idem for $a^{-1}b$ and $ab^{*}$.
\item $axa^{*} \in \V^n(R)$. 
\item If $\sm{ a & b \\ c & d } \in \GL(\Gamma_n(R))$, then $ax\overline{d} + b\overline{x}\overline{c} \in \V^n(R)$.
\end{enumerate} 
\end{lemma}

The first item is exactly \cite[Lemma 1.4]{1985ahlfors}. The second item is \cite[(3.9)]{ElsGrunMen}. The last item is proven in \cite[Proof of Theorem 3.7]{ElsGrunMen}. In \cite{ElsGrunMen}, everything is done for the field $\Q$, but the proofs remain valid for $\R$. Moreover, the last item is proven  for a matrix $\sm{a & b \\ c & d} \in \SL_+(\Gamma_n(\Q))$, but again the proof remains valid for a matrix in $\GL(\Gamma_n(\R))$.

Let $\Gamma_n(\Z)$ be the multiplicative monoid consisting of the elements of $\Cliff_n(\Z)$ that are products of non-zero elements of $\V^n(\R)$, i.e. $$\Gamma_n(\Z)=\Cliff_n(\Z) \cap \Gamma_n(\R).$$ 

Using the Euclidean norm, it is easy to see that the unit group $\U(\Gamma_n(\Z)) = \langle i_1, \ldots, i_{n-1} \rangle$ (see \Cref{structure_U_Gamma_n_Z} below for a presentation).

One defines $\SL_+(\Gamma_n(\Z))$ as the subgroup of $\SL_+(\Gamma_n(\R))$ consisting of all matrices with entries in $\Gamma_n(\Z) \cup \lbrace 0 \rbrace$. More precisely,
\begin{equation}\label{sl+Z}
\SL_+(\Gamma_n(\Z)) = \SL_+(\Gamma_n(\R)) \cap \Ma_2(\Gamma_n(\Z) \cup \lbrace 0 \rbrace).
\end{equation}
Note that $\SL_+(\Gamma_n(\Z))$ indeed is a group since the inverse of a matrix in $\SL_+(\Gamma_n(\Z))$ also belongs to $\SL_+(\Gamma_n(\Z))$ by \cite[Proposition 3.2]{ElsGrunMen}.

\subsection{The groups $\GE_2$ and $\E_2$}\label{defGe2E2} 

We recall from \cite{Cohn1} some definitions concerning the groups $\GE_2(R)$ and $\E_2(R)$ for an arbitrary unital ring $R$.
The group $\GE_2(R)$ is the group generated by all matrices 
\[ [\mu, \nu] = \left( \begin{matrix} \mu & 0 \\ 0 & \nu \end{matrix}\right) \textrm{ for } \mu, \nu \in \U(R), \qquad E(x) = \left( \begin{matrix} x & 1 \\ -1 & 0  \end{matrix}\right) \textrm{ for } x \in R. \]
For $\mu \in \U(R)$, put $D(\mu) = [\mu, \mu^{-1}]$. Define the groups $\D_2(R) = \langle [\mu, \nu] \ | \ \mu, \nu \in \U(R) \rangle$ and $\DE_2(R)= \langle D(\mu) \mid \mu \in \U(R) \rangle$.

In the group $\GE_2(R)$ the following relations hold, see \cite[(2.2)-(2.4)]{Cohn1}
\begin{align}
E(x)E(0)E(y) & = E(0)^2E(x+y), & & x, y \in R \tag{R1}\label{R1} \\
E(\mu)E(\mu^{-1})E(\mu) & = E(0)^2D(\mu), & & \mu \in \U(R) \tag{R2}\label{R2} \\
E(x)[\mu, \nu] & = [\nu, \mu] E(\nu^{-1}x\mu),  & &  x \in R, \; \mu, \nu \in \U(R) \tag{R3}\label{R3} \\
E(0)^2 & = D(-1)\tag{R4}.\label{R4}
\end{align}
 
The group $\E_2(R)$ is the group generated by all matrices
\[E(x) = \left( \begin{matrix} x & 1 \\ -1 & 0  \end{matrix}\right) \textrm{ for } x \in R. \]

Note that, by \eqref{R2}, $\E_2(R)$ contains the matrices $D(\mu)=\left[\mu,\mu^{-1}\right]$. Moreover, it can be shown, that $\E_2(R)$ is the group generated by the classical elementary matrices and that $E_2(D)$ is a subgroup of $\SL_2(D)$ for $D$ a division algebra.

\begin{definition}
The ring $R$ is said to be \textit{universal for $\GE_2$} if the relations \eqref{R1}-\eqref{R4} together with a set of defining relations in the group $\D_2(R)$, form a complete set of defining relations of $\GE_2(R)$. These relations  are called the \textit{universal relations}.
\end{definition}

In order to describe $\GE_2(R)$ and $\E_2(R)$ of some ring $R$ we will often try to provide a full presentation of these groups via abstract generators and relations. From now on we will adopt an abuse of notation: we will not make the distinction between the elements $E(x)$ and $[\mu, \nu]$ as matrices in $\GE_2(R)$ and the corresponding elements of an abstractly generated group which satisfy the same relations.

In \cite[Proposition 3.4]{abelianizationpaper}, it is shown that if $R$ is a ring such that there exists a set $\Phi$ of relations solely defined in $\E_2(R)$ such that $\Phi$ together with the universal relations yield a full list of relations for $\GE_2(R)$, then
\begin{equation}\label{theorem difference GE2 and E2}
\GE_2(R)/\E_2(R) \cong \mathcal{U}(R)^{ab}.
\end{equation}

The following definition is also taken from \cite{Cohn1}.

\begin{definition}
A ring $R$ is a \textit{$\GE_2$-ring} if $\GL_2(R) =\GE_2(R)$.
\end{definition}

The following proposition is folklore.   

\begin{proposition}\label{EuclideanGE2}
Left and Right Euclidean rings are $\GE_2$-rings. 
\end{proposition}

The discrete subrings $R$ of $\C$ that are $\GE_2$-rings have been classified in \cite[Theorem 3]{Dennis}. Namely, a discrete subring of $\C$ is a $\GE_2$-ring if and only if it is one of the following seven rings:
\begin{enumerate}
\item $\mathcal{I}_d$, the ring of integers in $\Q(\sqrt{-d})$, for $d=1,2,3,7,11$,
\item $\Z$,
\item $\Z[\sqrt{-3}]$.
\end{enumerate}

Note that the numbers $d=1,2,3,7,11$ are exactly those for which $\mathcal{I}_d$ is a Euclidean ring.

\section{The groups $\GE_2$ and $\E_2$ over the Clifford group}\label{introduction clifford}

In this section we define the groups $\GE_2$ and $\E_2$ over the Clifford group $\Gamma_n(\Z)$. As $\Gamma_n(\Z)$ is merely a group, but not a ring, one has to be careful while generalizing these concepts. 
Once we have well defined the groups $\GE_2(\Gamma_n(\Z))$ and $\E_2(\Gamma_n(\Z))$, we determine a presentation in terms of generators and relations of the latter. 
Recall that $i_0 = 1$ and set 
\begin{equation}\label{defB}
\mathcal{B} = \lbrace \pm i_h \mid 0 \leq h \leq n-1 \rbrace. 
\end{equation}

\begin{definition}\label{defGECliff}
The group $\GE_2(\Gamma_n(\Z))$ is the subgroup of $\GL(\Gamma_n(\Z))$ generated by all matrices 
\begin{align*} [\mu, \nu] & = \left( \begin{matrix} \mu & 0 \\ 0 & \nu \end{matrix}\right) \textrm{ for } \mu, \nu \in \U(\Gamma_n(\Z)) \textrm{ and } \mu\nu^{*} \in \R \setminus \lbrace 0 \rbrace , \\  E(x) & = \left( \begin{matrix} x & 1 \\ -1 & 0  \end{matrix}\right) \textrm{ for } x \in \V^n(\Z). \end{align*}
\end{definition}
Note that the matrices $\left[\mu, \nu\right] \in \GL_2(\Gamma_n(\Z))$ are all of the form 
$\left[\mu, \pm (\mu^*)^{-1}\right]$.
Indeed, as $\mu$ and $\nu$ are units in $\Gamma_n(\Z)$ and $\mu\nu^* \in \R$, $\mu\nu^* = \pm 1$. Hence
\begin{equation}\label{diagonalmatricGL}
\GE_2(\Gamma_n(\Z)) = \langle E(x), \left[\mu, \pm (\mu^*)^{-1}\right] \mid x \in \V^n(\Z), \mu \in \U(\Gamma_n(\Z)) \rangle.
\end{equation}

\begin{definition}\label{defECliff}
The group $\E_2(\Gamma_n(\Z))$ is the subgroup of $\GL(\Gamma_n(\Z))$ generated by all matrices 
$$ E(x) = \left( \begin{matrix} x & 1 \\ -1 & 0  \end{matrix}\right) \textrm{ for } x \in \V^n(\Z). $$
\end{definition}

\begin{remark}\label{isomE2normalE2Clifford}
The definition of $\GE_2(\Gamma_n(\Z))$ slightly differs from the definition of $\GE_2(R)$ for a ring $R$, due to the condition $\mu \nu^{*} \in \R$ for the entries $\mu$ and $\nu$ in the diagonal matrix $[\mu, \nu]$. For instance, for the case $n=2$, when $\Gamma_2(\Z) \cong \Z[i]$, the group $\GE_2(\Gamma_2(\Z))$ is isomorphic to a subgroup of index $2$ in $\GE_2(\Z[i]) \cong \GL_2(\Z[i])$ (defined according to \Cref{defGe2E2}). The condition $\mu\nu^{*} \in \R$ ensures precisely that $\GE_2(\Gamma_n(\R))$ is a group. Concerning the group $\E_2$, \Cref{defECliff} and the definition given in \Cref{defGe2E2} give the same groups when $\Gamma_n(\Z)$ is a commutative ring. Thus $\E_2(\Gamma_1(\Z)) \cong \E_2(\Z) \cong \SL_2(\Z)$ and $\E_2(\Gamma_2(\Z)) \cong \E_2(\Z[i]) \cong \SL_2(\Z[i])$. 
\end{remark}

In $\GE_2(\Gamma_n(\Z))$ the universal relations \eqref{R1}-\eqref{R4} hold, if we specify $x, y, \mu, \nu$ correctly for the Clifford group context. This means the following. 
\begin{itemize}
\item[\eqref{R1}] with $x, y \in \V^n(\Z)$. 
\item[\eqref{R2}] with $\mu \in \V^n(\Z) \cap \U(\Gamma_n(\Z)) = \mathcal{B}$ and $D(\mu)$ the matrix $\left[\mu, \mu^{-1} \right]$ with $\mu \in \mathcal{B}$.
\item[\eqref{R3}] with $\mu, \nu \in \U(\Gamma_n(\Z))$ and $\mu\nu^* \in \R \setminus \lbrace 0 \rbrace$. For the latter, note that if $[\mu, \nu] \in \GL(\Gamma_n(\Z))$, then also $[\overline{\nu}, \overline{\mu}] \in \GL(\Gamma_n(\Z))$. Hence by \Cref{simulVn}, for $x \in \V^n(\Z)$, $\overline{\nu}x\mu$ is an element in $\V^n(\Q)$. As $\overline{\nu}=\vert \nu \vert^2 \nu^{-1}$, also $\nu^{-1}x\mu$ is in $\V^n(\Q)$. As $\nu$ is by definition a unit in $\Gamma_n(\Z)$, the element $\nu^{-1}x\mu$ is contained in $\V^n(\Z)$.
\item[\eqref{R4}] stays the same relation. 
\end{itemize} 
In $\E_2(\Gamma_n(\Z))$, the relations \eqref{R1}, \eqref{R2} and \eqref{R4} hold with the elements quantified as in the previous paragraph. The equation \eqref{R3} specializes to 
\begin{align}
E(x)D(\mu) & = D(\mu^{-1}) E(\mu x\mu),  & &  x \in \V^n(\Z), \; \mu \in \U(\Gamma_n(\Z)). \tag{R3'}\label{R3'} 
\end{align}

Note that by \eqref{R2}, $\E_2(\Gamma_n(\Z))$ contains the elements $D(\mu) =  [ \mu, \mu^{-1}]$ with $\mu \in \mathcal{B}$. Nevertheless we will often add these diagonal matrices to the set of generators, as it simplifies the calculations.
Further define the following subgroups of $\GE_2(\Gamma_n(\Z))$ and $\E_2(\Gamma_n(\Z))$ respectively.
\begin{eqnarray}
\D_2(\Gamma_n(\Z)) & = & \langle [\mu, \nu] \mid \mu, \nu \in \U(\Gamma_n(\Z)), \mu\nu^* \in \R \setminus \lbrace 0 \rbrace \rangle. \label{diagonalgeneral}\\
\DE_2(\Gamma_n(\Z)) & = & \langle [\mu, \mu^{-1}] \mid \mu \in \mathcal{B} \rangle. \label{diagonalDE2}
\end{eqnarray}

The following lemma describes the group $\DE_2(\Gamma_n(\Z))$ more precisely. 

\begin{lemma}\label{description DE2}
The group $\DE_2(\Gamma_n(\Z))$ is isomorphic to $\U(\Gamma_n(\Z))$ for every $n \geq 1$. 
\end{lemma}

\begin{proof}
A matrix in $\DE_2(\Gamma_n(\Z))$ is a product of matrices of the form $\sm{ i_h & 0 \\ 0 & i_h^{-1}}$ for $1 \leq h \leq n-1$. Thus such a matrix has the form $\sm{
i_{h_1}i_{h_2}\ldots i_{h_m} & 0 \\ 0 & i_{h_1}^{-1}i_{h_2}^{-1}\ldots i_{h_m}^{-1}}.$
In general a matrix in $\DE_2(\Gamma_n(\Z))$ is hence of the form $ \sm{ a & 0 \\ 0 & (a^{*})^{-1}}$,
for some $a \in \U(\Gamma_n(\Z))$. 
Now let $\varphi$ be the following map:
\begin{align*}
\varphi: \U(\Gamma_n(\Z)) & \rightarrow \DE_2(\Gamma_n(\Z))  \\
a & \mapsto
\sm{a & 0 \\ 0 & (a^{*})^{-1}}.
\end{align*}
Since $^{*}$ and inversion are commuting anti-involutions, $\varphi$ is a homomorphism, which is clearly bijective.
\end{proof}

\begin{remark}\label{structure_U_Gamma_n_Z} We briefly elucidate the structure of the invertible elements of the monoid $\Gamma_n(\Z)$. For $n \geq 2$, the group \begin{align*} U = \U(\Gamma_n(\Z)) \cong \langle -1, i_1, i_2, ..., i_{n-1}  \ |\  & (-1)^2 = 1,\ -1 \text{ central}, \\ & i_h^2 = -1,\ i_hi_k = (-1)i_ki_h\ (1 \leq h < k \leq n-1)\ \rangle\end{align*}
has order $2^n$ and exponent $4$. 
Assume first $n$ odd. Then $\ZZ(U) = \langle -1 \rangle \cong C_2$ and $U/\ZZ(U) \cong C_2^{n-1}$, hence $U$ is extraspecial. Write $n = 2k + 1$, then, 
\[ \U(\Gamma_n(\Z)) \cong 
\begin{cases}  Q_8^{\Ydown k} & \text{  if } n \equiv 1, 3 \mod 8 \\ 
 Q_8^{\Ydown k-1} \Ydown D_8 & \text{  if } n \equiv 5, 7 \mod 8,
\end{cases}\]
where $\Ydown$ denotes a central product and $Q_8^{\Ydown r}$ denotes the iterated central product with $r$ factors isomorphic to $Q_8$.

Second assume that $n$ is even and set $z = i_1i_2\cdots i_{n-1}$. Then $\ZZ(U) = \langle\ -1,\ z\ \rangle$; the element $z \in U$ has order $2$ if $n \equiv 0 \bmod 4$ and order $4$ otherwise.  This gives the following.
\[ \U(\Gamma_n(\Z)) \cong 
\begin{cases} \langle i_1, ..., i_{n-2} \rangle \times \langle z \rangle \cong \U(\Gamma_{n-1}(\Z)) \times C_2  & \text{  if } n \equiv 0 \mod 4 \\ 
 \langle i_1, ..., i_{n-2}\rangle \Ydown \langle z \rangle \cong \U(\Gamma_{n-1}(\Z)) \Ydown C_4 & \text{  if } n \equiv 2 \mod 4,
\end{cases}\] 
In particular, $\U(\Gamma_2(\Z)) \cong C_4$, $\U(\Gamma_3(\Z)) \cong Q_8$, $\U(\Gamma_4(\Z)) \cong Q_8 \times C_2$, $\U(\Gamma_5(\Z)) \cong Q_8 \Ydown D_8$. \end{remark}

In \cite[page~160]{Cohn2} a description of the extra relations needed aside from the universal relations to obtain a presentation of $\GE_2(R)$ for $R$ an order in certain algebraic number fields has been obtained. In \cite[Proposition 3.1]{abelianizationpaper}, a quaternion variant of this theorem has been proven. In the same way it is possible to extend this theorem to the Clifford case, which is given below. We briefly point out the differences in the two proofs and sketch why it is possible to generalize the proof given in \cite{abelianizationpaper}.

\begin{proposition}\label{clifford cohn}
Let $K= \Q(\sqrt{-d})$ with $d\geq 0$ and $H = \qa{u}{v}{\Q}$ a totally definite quaternion algebra. Let $\O$ be an order in $K$, an order in $H$ or let $\O=\Gamma_n(\Z)$. Then, a complete set of defining relations for $\GE_2(\O)$ is given by the universal relations together with
\begin{equation}
(E(\overline{a}) E(a))^m = E(0)^{2} \quad \mbox{ if } 1 < |a| = \sqrt{m} < 2, \label{relation alpha}
\end{equation}
and $a \in \O$, or $a \in \V^n(\Z)$, if $\O = \Gamma_n(\Z)$.

A complete set of defining relations for $\E_2(\O)$ is given by \eqref{R1}, \eqref{R2}, \eqref{R3'}, \eqref{R4}, the relations in the group $\DE_2(\O)$ and relation \eqref{relation alpha}.
\end{proposition}

\begin{proof}
For $\O$ an order in $K$ or $H$, \Cref{clifford cohn} is exactly \cite[Proposition 3.1]{abelianizationpaper}. To prove \cite[Proposition 3.1]{abelianizationpaper}, one important auxiliary lemma is necessary. In the Clifford context, this lemma states that for $z, a \in \V^n(\R)$, $z \not= 0$ 
\begin{equation} |z-a| < 1 \quad \mbox{ if and only if }\quad \left|z^{-1}-\frac{1}{m-1} \overline{a}\right| < \frac{1}{m-1},\label{eq:norm_lemma clifford}\end{equation} where $1<|a| = \sqrt{m}$. Note that if $0\not=z \in \V^n(\R)$, then also $z^{-1} \in \V^n(\R)$. Also if $a \in \V^n(\R)$, then also $\frac{1}{m-1}\overline{a} \in \V^n(\R)$. Moreover the sum (or difference) of two elements in $\V^n(\R)$ is contained in $\V^n(\R)$. Hence the norm in \eqref{eq:norm_lemma clifford} is well defined. 
Now everything works as in \cite{abelianizationpaper}.

Concerning the second part of the defining relations of $\E_2(\O)$, for $\O$ an order in $K$ or $H$, this is proven in \cite[Theorem 2.14]{abelianizationpaper}. Although, $\Gamma_n(\Z)$ is not a ring, the proof stays valid.
\end{proof}

Next, we show the link between the groups $\E_2(\Gamma_n(\Z))$ and $\SL_+(\Gamma_n(\Z))$ for small values of $n$.
\begin{proposition}\label{E2=SL+}
For $1 \leq n \leq 4$,  $\E_2(\Gamma_n(\Z)) \cong \SL_+(\Gamma_n(\Z))$. 
\end{proposition}

\begin{proof}
Assume that $1 \leq n \leq 4$. By \Cref{defECliff}, the matrices generating $\E_2(\Gamma_n(\Z))$ have all entries in $\V^n(\Z)$. As $\Z$-module, $\V^n(\Z)$ is exactly $\Z^n$ and the norm on $\V^n(\Z)$ equals the Euclidean norm on $\Z^n$. This norm defines a kind of right Euclidean structure on $\Gamma_n(\Z)$ for $n \leq 4$. Indeed if $a, b \in \Gamma_n(\Z)$ such that $ab^{*} \in \V^n(\Z)$, then there exists $q \in \V^n(\Z)$ such that $\vert a - bq \vert \leq \vert b \vert$, where the inequality is strict for $n \leq 3$. To prove this, for $z \in \V^n(\R)$ define $\lceil z \rceil$ to be an element $z_0 \in \V^n(\Z)$ that minimizes $\vert z - z_0 \vert$. By identifying $\V^n(\Z)$ with $\Z^n$, a lattice in $\R^n$, one sees that $\vert z - \lceil z \rceil \vert \leq 1$ for every $z \in \V^n(\R)$ if $n \leq 4$ and the inequality is strict for $n \leq 3$. Let $a, b \in \Gamma_n(\Z)$ such that $ab^{*} \in \V^n(\Z)$. By \Cref{simulVn}, $a^{-1}b \in \V^n(\R)$ and thus also $b^{-1}a \in \V^n(\R)$ and if we denote $\lceil b^{-1}a \rceil$ by $q$, we have $\vert b^{-1}a - q \vert \leq 1$ or equivalently $\vert a - bq \vert \leq \vert b \vert$ and the inequality is strict for $n \leq 3$. This proofs the claim from above.

We can now proceed in the folklore way of the proof of \Cref{EuclideanGE2}. Let $\sm{a & b \\ c & d}$ be a matrix in $\SL_+(\Gamma_n(\Z))$. If $b = 0$, then \Cref{EuclideanGE2},
$$\begin{pmatrix}
a & 0\\ c & d\end{pmatrix} = [a,d]E(0)^3E(d^{-1}c).$$
By definition, $ad^{*}=1$ and hence $a, d^{*} \in \U(\Gamma_n(\Z))$ and thus, as in the proof of \Cref{description DE2}, the matrix $[a,d]$ is a product of generators of $\DE_2(\Gamma_n(\Z))$. The latter are  products of generators of $\E_2(\Gamma_n(\Z))$. Also by definition  $cd^{*} \in \V^n(\Z)$ and hence by \Cref{simulVn}, $c^{-1}d \in \V^n(\Q)$ and also $d^{-1}c \in \V^n(\Q)$. As $d$ is a unit, $d^{-1}c \in \V^n(\Z)$ and hence $E(d^{-1}c)$ is a generator of $\E_2(\Gamma_n(\Z))$.

Suppose now that $b \neq 0$ and suppose first that $n \leq 3$. Moreover, without loss of generality, we may suppose that $\vert b \vert \leq \vert a \vert$ (otherwise multiply the matrix by $E(0)$). By definition $ab^{*} \in \V^n(\Z)$ and thus, by the discussion above, there exists $q$ such that $\vert a- bq \vert < \vert b \vert$. Thus the matrix
\begin{equation*}
 \sm{a & b \\ c & d}E(0)^3E(-q)E(0)
 \end{equation*}
has $(-b, a-bq)$ as first row and hence the norm of the upper right entry is smaller than $b$. As the norm function only takes values in $\N$, this process ends up  in a matrix where the upper right entry is $0$.  Finally the first step allows to conclude. 

If $n=4$, it is sufficient to prove that the case $\vert a- bq \vert = \vert b \vert$ cannot occur. This is equivalent to $\vert b^{-1}a - \lceil b^{-1}a \rceil \vert = 1$. The latter means that $b^{-1}a$ is of the form $z + \frac{1+i_1+i_2+i_3}{2}$ for $z \in \V^n(\Z)$. Hence $a = b \cdot (z + \frac{1+i_1+i_2+i_3}{2})$. As the matrix has a kind of determinant, which is equal to $1$ (see \Cref{defcliffmoeb}), we get
$$b \cdot \left(\left(z + \frac{1+i_1+i_2+i_3}{2}\right)d^{*} - c^{*}\right) = 1.$$
Taking the norm on both sides, we get $\vert b \vert N =1$, where $N \in \N$. This contradicts the fact that $b^{-1}a$ is of the form $z + \frac{1+i_1+i_2+i_3}{2}$. This finishes the proof. 
\end{proof}

In \Cref{isomE2normalE2Clifford}, we stated that $\E_2(\Gamma_n(\Z))$ is isomorphic to the group $\E_2(R)$ if $\Gamma_n(\Z)$ is a commutative ring. For $n=4$, although $\Gamma_4(\Z)$ is not a ring, we get that $\E_2(\Gamma_4(\Z))$, and hence also $\SL_+(\Gamma_4(\Z))$, is isomorphic to the group $\E_2(\LL)$, where $\LL$ is the ring of Lipschitz quaternions. Recall that if $D=\left( \frac{-1, -1}{\Q} \right)$ is the ring of the ``standard'' quaternions over the rationals with standard basis $1, i, j, k$, then $\LL = \Z + \Z i + \Z j + \Z k \subseteq D$ is the ring of Lipschitz quaternions. 

\begin{corollary}\label{n3quat}
The group $\E_2(\Gamma_4(\Z))$ is isomorphic to $\E_2(\LL)$.
\end{corollary}  

\begin{proof}
It suffices to compare for both groups the presentations stated in  \Cref{clifford cohn}.
Clearly the correspondence $i_1 \mapsto i$, $i_2 \mapsto j$ and $i_3 \mapsto k$ determines a natural one-to-one correspondence between the generators of the groups $\E_2(\Gamma_n(\Z))$ and $\E_2(\LL)$ and this agrees in both groups with the relations \eqref{R1}--\eqref{R4}, 
The only relations which still need to be checked in order to have an isomorphism between the two groups are the relations defined by $\DE_2(\Gamma_4(\Z))$ and $\DE_2(\LL)$ and the relations \eqref{relation alpha} for both groups.
For the latter relations, it is sufficient to note that for $\Gamma_4(\Z)$ and $\LL$, the elements of norm $\sqrt{2}$ and $\sqrt{3}$ are also in correspondence. Finally, it is readily seen that $\DE_2(\LL) \cong Q_8 \times C_2 \cong \DE_2(\Gamma_4(\Z))$ and these isomorphisms are compatible with the correspondence stated above. 
\end{proof}

\Cref{n3quat} might seem confusing as $\Gamma_3(\Z)$ is isomorphic to $\LL$. However on the matrix level, one has to go up to $\Gamma_4(\Z)$ to find an isomorphism with the standard quaternions. This same behaviour has been spotted in \Cref{isocliffquat}.

\section{A finite presentation of $\GE_2$ and $\E_2$ over $\Gamma_n(\Z)$}\label{finite presentation clifford}

In the following we will describe a finite presentation for $\GE_2(\Gamma_n(\Z))$ and $\E_2(\Gamma_n(\Z))$. Recall that $\mathcal{B} = \lbrace \pm i_h \mid 0 \leq h \leq n-1 \rbrace$(see \eqref{defB}).

Let us come back to the relations \eqref{R1}-\eqref{R4}. From the relations \eqref{R1} to \eqref{R4} and the relations in the group $\D_2(R)$ it follows that the involution $E(0)^2=D(-1)$ is central in $\GE_2(R)$ and we denote it by $-I$. Moreover, \eqref{R1} is equivalent to
\begin{align}
E(x+y)=E(x)E(0)^{-1}E(y),\qquad x,y \in \V^n(\Z).\tag{R1'}\label{R1'}
\end{align}

The inverse of $E(x)$ (with $x \in \V^n(\Z)$) is given by the formula \begin{equation}E(x)^{-1} = E(0) E(-x) E(0),\tag{R5} \label{inv} \end{equation} which follows from \Cref{R1}.

We first prove that the group $\GE_2(\Gamma_n(\Z))$ is finitely generated. 

\begin{lemma}\label{finitelymanygen}
The group $\GE_2(\Gamma_n(\Z))$ is generated by $E(x)$ for $x \in \mathcal{B} \cup \lbrace 0 \rbrace$ and $\left[\mu, \pm \mu\right]$, for $\mu \in \mathcal{B}$. In particular, $\GE_2(\Gamma_n(\Z))$ is finitely generated. 
\end{lemma}

\begin{proof}
For non-negative integers $a_x$, \eqref{R1'} yields that 
$$E(\sum\limits_{x \in \mathcal{B}} a_{x} x)= E(0)\prod_{x \in \mathcal{B}} E(0)^{-1}E(a_{x} x) ,$$
Using \eqref{R1}, we get that $E(a_xx) \in \langle E(0),E(x) \rangle$ for every $x \in \mathcal{B}$. Hence, $E(\sum\limits_{x \in \mathcal{B}} a_{x} x) \in \langle E(0), E(x) \mid x \in \mathcal{B} \rangle$. 

Now we show that it is enough to consider the matrices $\left[\mu, \pm \mu\right]$, for $\mu \in \mathcal{B}$. By \eqref{diagonalmatricGL}, we know that the generators of $\GE_2(\Gamma_n(\Z))$ are all of the form $\left[\mu, \pm (\mu^*)^{-1} \right]$ for some $\mu \in \U(\Gamma_n(\Z))$. 
It is not hard to see that every matrix $[\mu, \pm (\mu^*)^{-1} ]$ is a product of matrices of the form $[i_h,\pm (i_h)^{-1}]$.
\end{proof}

\begin{remark} Note that in fact it is enough to consider $E(x)$ for $x \in \lbrace  i_h \mid 0 \leq h \leq n-1 \rbrace$ as generators, because $E(-x)$ can be described by $E(x)$ and \eqref{R1}. However, in a first instance, for the purpose of the relations, it is easier to keep these generators. 
\end{remark}

We now reduce the relations.

\begin{lemma}\label{R1forfinitegen}
The relation \eqref{R1} for $x, y \in \V^n(\Z)$ is equivalent to the  relations\begin{align}
E(x)E(0)^{-1}E(y) & = E(y)E(0)^{-1}E(x),  \tag{RC1} \label{CR1} \\
E(x)E(0)^{-1}E(-x) & = E(0),  \tag{RC1$\frac{1}{2}$} \label{CR1.5}
\end{align}
for $x,y \in \mathcal{B}\cup \lbrace 0 \rbrace$.
\end{lemma}

\begin{proof}
We first notice that \eqref{CR1.5} is exactly relation \eqref{R1} for $y=-x$. Furthermore \eqref{CR1} follows trivially from \eqref{R1'} as both the left and right hand side equal $E(x+y)$.
Next we show that \eqref{R1} for all $x,y \in \V^n(\Z)$ follows from \eqref{CR1} and \eqref{CR1.5} for $x,y \in \mathcal{B} \cup \lbrace 0 \rbrace$. For this notice that \eqref{R1} or \eqref{R1'} are equivalent to the relation
\begin{equation}\label{R1eq2}
\forall\ q,r,s,t \in \V^n(\Z) \qquad \textrm{ if } q+r=s+t \textrm{ then } E(q)E(0)^{-1}E(r)=E(s)E(0)^{-1}E(t).
\end{equation} 
Indeed one implication is trivial. For the converse, set $q=x$, $r=y$, $s=x+y$ and $t=0$. 

Thus we need to prove \eqref{R1eq2}. To do so, let $q,r,s,t \in \V^n(\Z)$ with $q+r=s+t$. Fix an order on $\mathcal{B}$. Write $q=\sum\limits_{x \in \mathcal{B}} q_{x} x$, $q_x \in \mathbb{Z}_{\geq 0}$. Then, in terms of generators, $E(q)$ can be written as
$$E(q)=E(0)\prod_{x \in \mathcal{B}} E(0)^{-1}e_{x},$$
where
\begin{equation*}
e_{x} = \underbrace{E(x)E(0)^{-1}E(x)\ldots E(0)^{-1}E(x)}_{q_{x} \textrm{ times}}
\end{equation*}

Similarly, we can split up $E(r), E(s)$ and $E(t)$ and by \eqref{CR1}, we can regroup the $E(x)$'s, $x \in \mathcal{B}$, for a fixed order of $\mathcal{B}$.
If $q+r=\sum_{x \in \mathcal{B}} c_x x =s+t$, then this decomposition is obtained from rearranging the decompositions of $q$ and $r$ and cancelling out those coefficients corresponding to elements of $\mathcal{B}$ with opposing sign.
Of course, the same holds for $s$ and $t$.
We can do the same on the level of matrices $E(x)$ via \eqref{CR1} for rearranging and via \eqref{CR1.5} for cancelling. Thus, using those two relations \eqref{CR1} and \eqref{CR1.5}, we obtain
$$E(q)E(0)^{-1}E(r)=
E(0)\prod_{x \in \mathcal{B}}E(0)^{-1}f_x
=E(s)E(0)^{-1}E(t),$$
where
$$f_x = \underbrace{E(x)E(0)^{-1}E(x)\ldots E(0)^{-1}E(x)}_{c_x \textrm{ times}}. $$
\end{proof}

\begin{lemma}\label{R3forfinitegen}
The relation \eqref{R3} for $x \in \V^n(\Z)$ and $\mu, \nu \in \U(\Gamma_n(\Z))$ is equivalent to \eqref{R3} for $x \in \mathcal{B}\cup \lbrace 0 \rbrace$ and $\mu \in \mathcal{B}$.
\end{lemma}
\begin{proof}
First note that \eqref{R3} is equivalent to
\begin{equation}
\left[\nu,\mu\right]^{-1}E(x)\left[\mu,\nu\right] = E(\nu^{-1}x\mu),  \qquad   x \in \V^n(\Z), \; \mu \in \U(\Gamma_n(\Z)). \tag{R3''}\label{R3''} 
\end{equation}
Suppose that \eqref{R3''} is true for a fixed $x,y \in \V^n(\Z)$ and arbitrary $\mu, \nu \in \U(\Gamma_n(\Z))$. Then, by \eqref{R1} (and thus \eqref{R1'}) and the fact that $E(0)^2$ is central and of order $2$,
\begin{eqnarray*}
\left[\nu,\mu\right]^{-1}E(x+y)\left[\mu,\nu\right] & = & 
\left[\nu,\mu\right]^{-1}E(x)E(0)^{-1}E(y)\left[\mu,\nu\right]\\
& = & E(0)^2  \left[\nu,\mu\right]^{-1}E(x)\left[\mu,\nu\right]\left[\mu,\nu\right]^{-1}E(0)\left[\nu,\mu\right]\left[\nu,\mu\right]^{-1}E(y)\left[\mu,\nu\right]\\
& = & E(0)^{2}E(\nu^{-1}x\mu)E(\mu^{-1}0\nu)E(\nu^{-1}y\mu)\\
& = & E(\nu^{-1}x\mu)E(0)^{-1}E(\nu^{-1}y\mu)\\
& = & E(\nu^{-1}x\mu+\nu^{-1}y\mu)\\
& = & E(\nu^{-1}(x+y)\mu)
\end{eqnarray*}
By induction the relation \eqref{R3} for $x \in \V^n(\Z)$ and arbitrary $\mu, \nu \in \U(\Gamma_n(\Z))$ is equivalent to \eqref{R3} for $x \in \mathcal{B}$ and arbitrary $\mu, \nu \in \U(\Gamma_n(\Z))$.

The only thing missing is to show that it is enough to consider \eqref{R3} for $\mu \in \mathcal{B}$. Suppose \eqref{R3} for $i_h \in \mathcal{B}$ and we want to prove \eqref{R3} for general $\mu \in \U(\Gamma_n(\Z))$. By \eqref{diagonalmatricGL}, a matrix $\left[\mu, \nu \right] \in \GE_2(\Gamma_n(\Z))$ is of the form $[\mu, \pm (\mu^*)^{-1}]$ with $\mu \in \U(\Gamma_n(\Z))$. Write $[\mu, \pm (\mu^*)^{-1}] = [i_{h_1}, \pm (i_{h_1})^{-1}] \ldots [i_{h_m}, (i_{h_m})^{-1}]$, where $\mu = i_{h_1}\cdots i_{h_m}$ for some $m \in \N$ and some choice of $\pm$ in the first generator, depending on $m$. Then
\begin{align*}
E(x)[\mu, \pm (\mu^*)^{-1}] & = E(x)[i_{h_1}, \pm (i_{h_1})^{-1}] \ldots [i_{h_m}, (i_{h_m})^{-1}]\\
& = [\pm i_{h_1}^{-1}, i_{h_1}]E(\pm i_{h_1}xi_{h_1})\ldots [i_{h_m},(i_{h_m})^{-1}]\\
& = [\pm i_{h_1}^{-1}, i_{h_1}] \ldots [i_{h_m}^{-1},i_{h_m}]E(\pm i_{h_m} \ldots i_{h_1}xi_{h_1} \ldots i_{h_m})\\
& = [\pm (\mu^*)^{-1}, \mu]E(\pm \mu^* x \mu). \qedhere
\end{align*} 
\end{proof}

The next theorem follows immediately from the previous lemmas. It actually states that the universal relations for $\GE_2(\Gamma_n(\Z))$ may be reduced to the universal relations for the generating set $\mathcal{B}$. Moreover the same is true for the group $\E_2(\Gamma_n(\Z))$.

\begin{theorem}\label{gen and rel for GECliff}\label{gen and rel for E2Cliff}
The group $\GE_2(\Gamma_n(\Z))$ (respectively $\E_2(\Gamma_n(\Z))$) is generated by $E(x)$, with $x \in \mathcal{B}  \cup \lbrace 0 \rbrace,$ and $[\mu, \pm \mu]$ (respectively $D(\mu)$) for $\mu \in\mathcal{B}$. \\
A complete set of defining relations for $\GE_2(\Gamma_n(\Z))$ is given by
\begin{align}
E(x)E(0)^{-1}E(y) & = E(y)E(0)^{-1}E(x), & & x,y \in \mathcal{B}\cup \lbrace 0 \rbrace, \label{RC1}\tag{RC1} \\
E(x)E(0)^{-1}E(-x) & = E(0), & & x \in \mathcal{B}\cup \lbrace 0 \rbrace, \label{RC1.5theorem}\tag{RC1$\frac{1}{2}$} \\
E(\mu)E(\mu^{-1})E(\mu) & = E(0)^2D(\mu), & & \mu \in \mathcal{B}, \label{RC2}\tag{RC2} \\
E(x)[\mu, \pm \mu^{-1}] & = [\pm \mu^{-1}, \mu] E(\pm \mu x\mu),  & &  x \in \mathcal{B} \cup \lbrace 0 \rbrace, \; \mu \in \mathcal{B},\label{RC3} \tag{RC3}\\
E(0)^2 & = D(-1),\label{RC4}\tag{RC4}\\
\textrm{defining relations } & \textrm{in } \D_2(\Gamma_n(\Z)), & & \label{RC6}\tag{RC5}\ \\
(E(\overline{a}) E(a))^m & = E(0)^{2} & & a \in \V^n(\Z) \mbox{ such that } 1 < |a| = \sqrt{m} < 2. \tag{RC6}\label{RC7}
\intertext{The group $\E_2(\Gamma_n(\Z))$ is defined by the same relations, except that \eqref{RC3} and \eqref{RC6} are replaced by}
E(x)D(\mu)&=D(\mu^{-1})E(\mu x \mu), & &  x \in \mathcal{B}\cup \lbrace 0 \rbrace, \mu \in \mathcal{B} \label{RC3'}\tag{RC3'}\\
\textrm{Defining relations } & \textrm{in } \DE_2(\Gamma_n(\Z)). & & \label{RC6'}\tag{RC5'}
\end{align}
In particular $\GE_2(\Gamma_n(\Z))$ and $\E_2(\Gamma_n(\Z))$ are finitely presented.
\end{theorem}

\begin{remark}\label{finitelypresentedgeneralR} \Cref{gen and rel for GECliff} stays valid for a special class of unital rings $R$ (almost universal rings, see \cite[Definition 2.13]{abelianizationpaper} for more details) which are finitely generated and free as a $\Z$-module with basis $\mathcal{B}$, with $x,y \in \mathcal{B}$ and $\mu, \nu \in \U(R)$.  Moreover a complete set of relations for $\GE_2(R)$ is given by the universal relations and a set $\Phi$ of relations solely defined in $\E_2(R)$.
If both $\U(R)$ and the set $\Phi$ are finite, then $\GE_2(R)$ and $\E_2(R)$ are finitely presented. 
\end{remark}

Although there are several elements in $\V^n(\Z)$ with norms between $1$ and $2$, the next theorem shows that nevertheless the universal relations are enough to give a presentation of $\GE_2(\Gamma_n(\Z))$.

\begin{theorem}\label{universalcomplete}
The universal relations give a complete set of relations for $\GE_2(\Gamma_n(\Z))$ and $\E_2(\Gamma_n(\Z))$.
\end{theorem}

\begin{proof}
We only need to prove that the relations~(\ref{RC7}) can be deduced from the universal relations. Then \Cref{clifford cohn} allows to conclude. 

Analogously to \cite[(2.9)]{Cohn1}, one can derive the following crucial formula from the universal relations.
\begin{equation}\label{2_9'}
 E(x)E(\alpha)E(y) \; = \; E(x - \alpha^{-1})D(\alpha)E(y - \alpha^{-1}), \qquad x, y \in \V^n(\Z), \alpha \in \mathcal{B}. 
\end{equation}

We first consider the elements $a \in \V^n(\Z)$ appearing in $E(a)$ and that are of norm $2$. They are necessarily of the form $\gamma + \delta$ with $\gamma, \delta \in \mathcal{B} = \lbrace  \pm i_h \mid 0 \leq h \leq n-1 \rbrace$ and $\gamma \neq \pm \delta$. We can assume, without loss of generality, that $\gamma \neq \pm 1$ (as $\gamma \neq \pm \delta$) and thus $\gamma^2 = -1$.
Furthermore, $\delta \gamma \delta = \gamma$, $\overline{\gamma} = \gamma^{-1} = -\gamma$ and $\overline{\delta} = \delta^{-1}$. In what follows, the central element $E(0)^2$ will frequently be replaced by a minus sign in front of the product.
Consider $\left( E(\overline{\gamma+\delta})E(\gamma + \delta) \right) ^2$. 

\begin{eqnarray*}
 \left(E(\overline{\gamma+\delta})E(\gamma + \delta) \right)^2 & = & \left(E(\overline{\gamma}+\overline{\delta})E(\gamma + \delta)\right) ^2\\
& \stackrel{\eqref{RC1}}{=} & \left(E(\overline{\gamma})E(0)E(\overline{\delta})E(\gamma)E(0)E(\delta)\right)^2 \\ 
& \stackrel{\eqref{2_9'}}{=}& \left( E(\gamma^{-1}) E(-\delta)D(\delta^{-1})E(\gamma - \delta)E
(0)E(\delta)\right)^2 \\
& \stackrel{\eqref{RC1}}{=}& \left( E(\gamma^{-1})E(-\delta)D(\delta^{-1})E(\gamma)\right)^2\\ 
& \stackrel{\eqref{RC3'}}{=}&\left(D(\delta^{-1}) E(\delta \gamma^{-1} \delta )E(-\delta^{-1})E(\gamma)\right)^2 \\
& \stackrel{\eqref{RC3'}}{=}&D(\delta^{-1})D(\delta) E(\gamma^{-1})E(-\delta)E(\delta^{-1}\gamma \delta^{-1}) E(\delta \gamma^{-1} \delta )E(-\delta^{-1})E(\gamma) \\
& \stackrel{\eqref{RC6'}}{=}& E(\gamma^{-1})E(-\delta)E(\gamma ) E(\gamma^{-1} )E(-\delta^{-1})E(\gamma) \\
&  \stackrel{\eqref{2_9'}}{=}& E(\gamma^{-1})E(-\delta-\gamma^{-1})D(\gamma ) E(0)E(-\delta^{-1})E(\gamma)\\
& \stackrel{\eqref{RC3'}}{=}& D(\gamma )E(\gamma^{-3})E(-\gamma\delta\gamma-\gamma) E(0)E(-\delta^{-1})E(\gamma)\\ 
& \stackrel{\eqref{RC1}}{=}& -D(\gamma )E(\gamma)E(-\gamma\delta\gamma-\gamma-\delta^{-1})E(\gamma).
\end{eqnarray*}
As $- \gamma \delta \gamma = \delta^{-1}$, the last right hand side simplifies to $$-D(\gamma )E(\gamma)E(-\gamma)E(\gamma) = -D(\gamma )E(\gamma)E(\gamma^{-1})E(\gamma) \stackrel{\eqref{RC2}}{=} D(\gamma)^2 = -I. $$
If $a$ has norm $3$, the computations are similar, but longer. The details can be found in \cite[Proof of Theorem 4.4.4]{Doryanthesis}.

The reductions in both cases were actually done completely within the group $\E_2(\Gamma_n(\Z))$, by which we mean that we used relations which are also true in $\E_2(\Gamma_n(\Z))$ and not just in $\GE_2(\Gamma_n(\Z))$. For example, we used \eqref{RC3'} instead of \eqref{RC3} and \eqref{RC6'} instead of \eqref{RC6}. This also implies that \eqref{RC7} may be dropped in the relations of $\E_2(\Gamma_n(\Z))$.
\end{proof}

\begin{remark}\label{dropRC6}
Note that the relations~(\ref{RC7}) were never used in the reduction processes in the proofs of \Cref{finitelymanygen} to \Cref{R3forfinitegen}. Thus the relations~(\ref{RC7}) may also be dropped in \Cref{gen and rel for GECliff}.
\end{remark}

The idea of the proof of \Cref{universalcomplete} is not unique to the case of Clifford algebras.
An important aspect of the proof is that the ``basis'' of $\Gamma_n(\Z)$ consists of units.
Hence, with some extra work, \Cref{universalcomplete} remains valid for quaternion orders containing a basis of units.
For example the bases of the Lipschitz and Hurwitz order $\LL$ and $\O_2$ in $\qa{-1}{-1}{\Q}$ and the maximal order $\O_3$ in $\qa{-1}{-3}{\Q}$ stated in \cite[(3.20)]{abelianizationpaper} consist of units. Thus by a similar, but more complicated, proof as above, one can extend \Cref{universalcomplete} to these orders.

However the universal relations do not always give a complete set of relations of $\GE_2(\O)$ for $\O$ an order in a quaternion algebra. A non-exotic example of this is given in \Cref{appendix}.

\section{The group $\E_2(\Gamma_n(\Z))$ is an amalgamated product}\label{mainsection}

In this section we prove that $\E_2(\Gamma_n(\Z))$ is a non-trivial amalgamated product for every $n \in \N$. This is well-known for $n=1$  and $n=2$. Indeed in \cite[I.4.2 Example c]{Serre}, Serre showed that $\E_2(\Gamma_1(\Z))=\SL_2(\Z)$ is a non-trivial free product with amalgamation and in \cite[Theorem 4.4.1]{FineBook}, it is shown that $\E_2(\Gamma_2(\Z)) \cong \SL_2(\Z[i])$ is a non-trivial free product with amalgamation. In this section, we will show the amalgamation for $n \geq 3$. The work is inspired by \cite[Section 4]{FineBook}.

\Cref{gen and rel for E2Cliff} and \Cref{universalcomplete} give a first presentation for $\E_2(\Gamma_n(\Z))$. This is summarized in the following lemma. 

\begin{lemma}\label{firstpresent}
For every $n \geq 1$, the group $\E_2(\Gamma_n(\Z))$ is generated by $E(x)$ for $x \in \lbrace 0, \pm i_h \mid 0 \leq h \leq n-1 \rbrace$ and $D(\mu)$ for $\mu \in  \lbrace  \pm i_{h} \mid 0 \leq h \leq n-1 \rbrace$ with the following relations.
\begin{enumerate}
\item $E(0)^4 = I,\ E(0)^2 \textrm{ central}$\label{(1)}
\item $E(i_h)E(0)E(-i_h) =E(0)^3$ for $ 0 \leq h \leq n-1 $\label{(2)}
\item $E(i_h)E(0)E(i_k)  = E(i_k)E(0)E(i_h)$ for $0 \leq h < k \leq n-1$\label{(3)}
\item $D(1)  = E(0)^2E(1)^3$\\
$D(-1)  = E(0)^2E(-1)^3$\\
$D(i_h)  = E(0)^2E(i_h)E(-i_h)E(i_h)$ for $1 \leq h \leq n-1$\label{(4)}
\item $E(0)D(-1)  = D(-1)E(0)$ \\
$E(0)D(i_h)  = D(-i_h)E(0)$ for $1 \leq h \leq n-1$ \\
$E(1)D(-1)  = D(-1)E(1)$ \\
$E(1)D(i_h) = D(-i_h)E(-1)$ for $1 \leq h \leq n-1$\\
$E(i_h)D(-1) = D(-1)E(i_h)$  for $1 \leq h \leq n-1$ \\
$E(i_h)D(i_h)  = D(-i_h)E(-i_h)$ for $1 \leq h \leq n-1$ \\
$E(i_h)D(i_k)  = D(-i_k)E(i_h)$ for $1 \leq h \neq k \leq n-1$ \label{(5)}
\item 
Defining relations given by $\DE_2(\Gamma_n(\Z)) \cong \U(\Gamma_n(\Z))$.\label{(6)}
\end{enumerate}
\end{lemma}

\begin{proof}
It is straightforward that the relations \eqref{(1)} to \eqref{(6)} follow from the relations in \Cref{gen and rel for E2Cliff}. 

For the converse, note that the relations \eqref{(1)} to \eqref{(6)} are in fact just the relations from \Cref{gen and rel for E2Cliff}, but where we dropped the cases $\mu = - i_h$ in $D(i_h)$ for $1 \leq h \leq n-1$ and $x=-i_h$ in $E(x)$ for $1 \leq h \leq n-1$. The relations for those cases can be deduced from the other relations by using
$$E(x) = E(0)E(-x)^{-1}E(0),\quad x \in \lbrace i_h \mid 0 \leq h \leq n-1 \rbrace,$$
which is implied by \eqref{(2)} and
$$E(0)D(-i_h)=D(i_h)E(0), \quad  0 \leq h \leq n-1,$$
which is implied by \eqref{(5)} and \eqref{(6)}.
Moreover the relations involving the generator $D(1)$ may be dropped, as $D(1)$ is the neutral element of the group by \eqref{(6)}.

Finally by \Cref{dropRC6}, the relations \eqref{RC7} can be dropped anyway. 
\end{proof}

We reduce this list of relations and generators further in the following lemma. Remark that we technically do not need to add the element $-I$ as generator, but it will make some relations easier to write down.

\begin{lemma}\label{secondpresent}
Let $n \geq 1$. The group $\E_2(\Gamma_n(\Z))$ is generated by $-I, E(0), E(1), E(i_h), D(i_h)$ for $1 \leq h \leq n-1$ with the following relations.
\begin{enumerate}[(a)]
\item $(-I)^2= I, -I \textrm{ central}$\label{(a)}
\item $E(0)^2=D(i_h)^2=(D(i_h)E(0))^2=E(1)^3=-I \textrm{ and }(D(i_h)E(i_h))^3=I$ for $1 \leq h \leq n-1$\label{(b)}
\item
$E(1)E(0)D(i_1)E(1)D(i_1)E(0)=I$\\
$E(i_h)E(0)D(i_h)E(i_h)D(i_h)E(0)=I$ for $1 \leq h \leq n-1$\label{(c)}
\item $E(1)E(0)E(i_h)  = E(i_h)E(0)E(1)$ for $1 \leq h \leq n-1$\\
$E(i_h)E(0)E(i_k)  = E(i_k)E(0)E(i_h)$ for $1 \leq h < k \leq n-1$\label{(d)}
\item $D(i_1)E(1)D(i_1)=D(i_2)E(1)D(i_2)=\ldots =D(i_{n-1})E(1)D(i_{n-1})$\label{(e)}
\item $D(i_1)E(i_h)D(i_1)=\ldots = D(i_{h-1})E(i_h)D(i_{h-1})= D(i_{h+1})E(i_h)D(i_{h+1})= \ldots =$\\
$ D(i_{n-1})E(i_h)D(i_{n-1}) = E(i_h)$ for $1 \leq h \leq n-1$\label{(f)}
\item $D(i_h)D(i_k)D(i_h)=D(i_k)$ for $1 \leq h<k \leq n-1$\label{(g)}
\end{enumerate}
\end{lemma}

\begin{proof}
First note that the previous lemma is true for $n=1$. For the rest of the proof we suppose $n \geq 2$. First we show that we can drop some of the generators. By relations \eqref{(6)} from \Cref{firstpresent}, $D(-i_h)=D(i_h)^{-1}$ for $1 \leq h \leq n-1$. Hence the generators $D(-i_h)$, with $0 \leq h \leq n-1$, may be dropped. By relations \eqref{(5)}, 
\begin{align*}
E(-1) & = D(-i_1)^{-1}E(1)D(i_1) = D(i_1)E(1)D(i_1)\\
E(-i_h) & = D(-i_h)^{-1}E(i_h)D(i_h) = D(i_h)E(i_h)D(i_h)
\end{align*}
Hence we can also drop these generators by replacing them by the above expression. We set $-I=E(0)^2$ and add it to the list of generators. Finally $D(-1)$ may be written by the second relation in \eqref{(4)} as $E(0)^2D(i_1)E(1)^3D(i_1)$ and hence it may also be dropped. This gives the list of generators. 

Now we turn to the relations. Clearly relation \eqref{(1)} and \eqref{(a)} are equivalent. We consider the relations in \eqref{(b)}. The first one is the definition of $-I$. The relations $D(i_h)^2=-I$ for $1 \leq h \leq n-1$ are parts of the relations given by $\DE_2(\Gamma_n(\Z))$, i.e. \eqref{(6)}. The next relations are equivalent with the second relations in \eqref{(5)}. 

Finally, the first relation in \eqref{(4)} is equivalent with $E(1)^3=-I$. Note that this also gives us that $D(-1)=E(0)^2=-I$. Replacing $E(-i_h)$ for $1 \leq h \leq n-1$ by their expressions in relations \eqref{(4)} gives, 
$$D(i_h)=E(0)^2E(i_h)D(i_h)E(i_h)D(i_h)E(i_h).$$
Multiplying the left hand side and the right hand side by $D(i_h)$ gives the relation in \eqref{(b)}.
Relations \eqref{(c)} are exactly relations \eqref{(2)} where $E(-1)$ and $E(-i_h)$ for $1 \leq h \leq n-1$ have been replaced by their new expressions. Also relations \eqref{(d)} and \eqref{(3)} are exactly the same. 
Relations \eqref{(e)} come from the fact that $E(-1)$ can be expressed in three different forms from relations \eqref{(5)}. 
In the same way, relations \eqref{(f)} are derived from the fact that $E(i_h)$ for $1 \leq h \leq n-1$ have different expressions in relations \eqref{(5)}. 
Finally relations \eqref{(g)} represent parts of the relations defining $\DE_2(\Gamma_n(\Z))$. 

This shows that relations \eqref{(a)}-\eqref{(g)} are derived from \eqref{(1)}-\eqref{(6)}. To show the converse, we only need to prove that the second relation of \eqref{(4)}, the relations of \eqref{(5)} involving $D(-1)$ and the relations \eqref{(6)} can be derived from \eqref{(a)}-\eqref{(g)}. 
For the second relation of \eqref{(4)}, since $D(-1)$ is renamed $-I$ and $E(0)^2=-I$ this is equivalent to showing that $E(-1)^3 = I$. Using the form of $E(-1)$ above, $E(-1)^3 = D(i) E(1)^3 D(i) = -D(i)^2 = I$. 
Consider now the relations in \eqref{(5)}. These are just a consequence of the fact that $D(-1)=-I$ which is central by \eqref{(a)}. 
Finally the relations $D(i_h)^2 =-I$ for $1 \leq h \leq n-1$ together with \eqref{(g)} give all the relations defining $\U(\Gamma_n(\Z))$, and by \Cref{description DE2}, the latter is isomorphic with $\DE_2(\Gamma_n(\Z))$. Thus relations \eqref{(6)} can be derived from \eqref{(a)}-\eqref{(g)}.
\end{proof}

Continuing our quest for an amalgamation, we follow the strategy of Fine in \cite{FineBook} and rewrite the generators as follows

\begin{tabular}{lll}
$J  = -I = \begin{pmatrix}
-1 & 0 \\ 0 & -1
\end{pmatrix},$ 
& $A = E(0)  = \begin{pmatrix}
0 & 1 \\ -1 & 0 
\end{pmatrix},$
& $T_1  = E(0)E(1)^{-1}  = \begin{pmatrix}
1 & 1 \\ 0 & 1 
\end{pmatrix},$\\ 
$T_{i_h}  = E(0)E(i_h)^{-1} = \begin{pmatrix}
1 & i_h \\ 0 & 1 
\end{pmatrix},$
& $L_{i_h}  = D(i_h)  = \begin{pmatrix}
i_h & 0 \\ 0 & -i_h 
\end{pmatrix},$& for $ 1 \leq h \leq n-1$.
\end{tabular}

Then we get the following abstract presentation for $\E_2(\Gamma_n(\Z))$.

\begin{lemma}\label{firstabstractpresent}
For $n \geq 1$, $\E_2(\Gamma_n(\Z))$ is generated by $ J, A, T_1, T_{i_h}, L_{i_h}$ for $ 1 \leq h \leq n-1$ and the relations are given by
\begin{enumerate}[(i)]
\item $J^2= I, J \textrm{ central }$ \label{(i)}
\item $A^2=L_{i_h}^2=(AL_{i_h})^2=J$  for $\ 1 \leq h \leq n-1$\label{(ii)}
\item $(T_1A)^3=(T_{i_h}L_{i_h}A)^3=I$  for $1 \leq h \leq n-1$ \label{(iii)}
\item $\left[T_1,T_{i_h}\right]=\left[T_{i_h},T_{i_k}\right]=I$, for $1 \leq h <k \leq n-1$\label{(iv)}
\item $(L_{i_1}^{-1}T_1)^2=(L_{i_h}^{-1}T_{i_h})^2=J$ for $1 \leq h \leq n-1$\label{(v)}
\item $L_{i_1}T_1L_{i_1}=\ldots =L_{i_{n-1}}T_1L_{i_{n-1}}$\label{(vi)}
\item $\left[L_{i_h},T_{i_k}\right]=I$ for $ 1 \leq h \neq k \leq n-1$\label{(vii)}
\item $ (L_{i_h}L_{i_k})^2=J$ for $1 \leq h<k \leq n-1$ \label{(viii)}
\end{enumerate}
\end{lemma}

\begin{proof}
Clearly \eqref{(i)} and \eqref{(a)} are exactly the same. Relations \eqref{(ii)} and \eqref{(iii)} are just a rewriting of the relations \eqref{(b)} from \Cref{secondpresent}.
Relations \eqref{(iv)} and \eqref{(v)} are equivalent with relations \eqref{(d)} and \eqref{(c)} respectively.
Relation \eqref{(vi)} corresponds to \eqref{(e)} and \eqref{(vii)} corresponds to the relations in \eqref{(f)}. Indeed, consider the relation $\left[L_{i_h},T_{i_k}\right]=I$. This gives $D(i_h)E(0)E(i_k)^{-1}D(i_h)E(i_k)E(0)=I$. Using that $D(i_h)$ and $E(0)$ anticommute and that $E(0)^2=J$, we get that $D(i_h)E(i_k)^{-1}D(i_h)E(i_k)=I$. This is clearly equivalent with $D(i_h)E(i_k)^{-1}D(i_h)=E(i_k)^{-1}$. By taking inverses we are done. 
Finally relations \eqref{(viii)} are exactly the relations in \eqref{(g)}. 
\end{proof}

\begin{remark} \label{E_2LL_2abelianization}
From the presentation in \Cref{firstabstractpresent} one can easily determine the abelianization of $\E_2(\Gamma_n(\Z))$. One obtains
\begin{align*}
\E_2(\Gamma_1(\Z))^{ab} & \cong C_{12},\\
\E_2(\Gamma_n(\Z))^{ab} & \cong C_2^{n}  \textrm{ for } n \geq 2.
\end{align*}
In particular we recover the well-known results that $\SL_2(\Z)^{ab} \cong C_{12}$ and $\SL_2(\Z[i])^{ab} \cong C_2 \times C_2$ \cite[Corollary~5.2]{Swan}. It also shows, by \Cref{n3quat}, that $\E_2(\LL)^{ab} \cong \E_2(\Gamma_4(\Z))^{ab} \cong C_2^4$.
Note that this has been proven with different methods in \cite{abelianizationpaper}. 
\end{remark}

We rewrite the generators one last time. Inspired by \cite{FineBook} and \cite{MacWatWie}, we set
$$j = J, \ a=A,\ b_{i_h}=AL_{i_h},\   c=T_1A, d_{i_h}=T_{i_h}L_{i_h}A,$$
for $1 \leq h \leq n-1$.
We are able to rewrite the presentation of $\E_2(\Gamma_n(\Z))$ given these generators. The proof of the following lemma works along the same lines as the proofs before. 

\begin{lemma} \label{presentationOL}
For $n \geq 1$,
\begin{eqnarray*}
\E_2(\Gamma_n(\Z)) & = &\left\langle j,a,b_{i_h},c,d_{i_h}, \textrm{ for } 1 \leq h \leq n-1 \mid \right.\\ 
&& j^2= 1, j \textrm{ central },\\
&&a^2=b_{i_h}^2=(ab_{i_h})^2=j \textrm{ for } 1 \leq h \leq n-1,\\
&&c^3=d_{i_h}^3=1 \textrm{ for } 1 \leq h \leq n-1,\\
&& (b_{i_h}c)^2=(ad_{i_h})^2=(d_{i_h}c^{-1})^2=j \textrm{ for } 1 \leq h \leq n-1,\\
&&(b_{i_h}b_{i_k})^2=j \textrm{ for } 1 \leq h<k \leq n-1,\\
&&(b_{i_h}d_{i_k})^2=j \textrm{ for } 1 \leq h \neq k \leq n-1,\\
&&(d_{i_h}d_{i_k}^{-1})^2=1 \textrm{ for } 1 \leq h<k \leq n-1\rangle. 
\end{eqnarray*}
\end{lemma}

Finally, the presentation in \Cref{presentationOL} has the perfect form to prove our main theorem. 

\begin{theorem}\label{prop:E_2_O_L_is_amalgam}
For $n \geq 1$, the group $\E_2(\Gamma_n(\Z))$ is a non-trivial amalgamated product. Furthermore, for $n \geq 2$, the group over which the product is amalgamated is $\E_2(\Gamma_{n-1}(\Z))$. 
\end{theorem}
\begin{proof}
For $n \geq 2$, this is a consequence of \Cref{disjoint generators amalgam} and \Cref{presentationOL}. It suffices to take $A= \{j,a,b_{i_1},\ldots, b_{i_{n-1}},c,d_{i_1}, \ldots, d_{i_{n-2}}\} $, $B = \{j,a,b_{i_1},\ldots,b_{i_{n-2}},c,d_{i_1}, \ldots, d_{i_{n-1}}\} $ and $C=\{j,a,b_{i_h},c,d_{i_h} \mid 1 \leq h \leq n-2\}$. The group over which the amalgamated product is taken, is a subgroup of $\E_2(\Gamma_n(\Z))$ generated by $j, a, b_{i_h},c,d_{i_h}$ for $1 \leq h \leq n-2$. By \Cref{presentationOL}, this is clearly $\E_2(\Gamma_{n-1}(\Z))$.  

For $n=1$, it suffices to take $A= \{j,a\} $, $B = \{j,c\} $ and $C=\{j\}$. Then $\E_2(\Z)$ is an amalgamated product over the group $C_2$. 
\end{proof}

\Cref{prop:E_2_O_L_is_amalgam} confirms results that are well-known in case $n=1$ and  $n=2$. In \cite[I.4.2 Example c)]{Serre} it is shown that $\SL_2(\Z)$ is the amalgamated product $C_4 \ast_{C_2} C_6$, which is what we obtain for $n=1$. In \cite[Theorem 4.4.1]{FineBook}, it is shown that the group $\PSL_2(\Z[i])$ is the amalgamated product $G_1 \ast_{\PSL_2(\Z)} G_2$. By \Cref{amalgamconserved}, we get an amalgamated product for $\SL_2(\Z[i])$, which is exactly the one we obtain for $n=2$, cf.\ also \Cref{isomE2normalE2Clifford}.

By \Cref{E2=SL+}, this implies an amalgam structure for $\SL_+(\Gamma_n(\Z))$ for small values of $n$.

\begin{corollary}\label{SL_2L_2isamalgam}
For $n \leq 4$, the group $\SL_+(\Gamma_n(\Z))$ is a non-trivial amalgamated product over $\SL_+(\Gamma_{n-1}(\Z))$.  
\end{corollary}

\Cref{prop:E_2_O_L_is_amalgam} provides us thus with an alternative proof of \cite[Theorem 4.4.1]{FineBook}, which is the case $n = 2$.
Moreover, the group described in \cite[Section 6]{MacWatWie}, which is $\SL_+(\Gamma_3(\Z))$, is thus an amalgamated product over $\SL_2(\Z[i])$ and it is the amalgamated subgroup of the over group $\E_2(\Gamma_4(\Z)) = \SL_+(\Gamma_4(\Z)) \cong \E_2(\LL)$.
\Cref{prop:E_2_O_L_is_amalgam} says that this behaviour, i.e. the previous group is the amalgamated subgroup of the next group, continues on for the series of groups $(\E_2(\Gamma_n(\Z)))_{n\in \N\setminus \{0\}}$.
Remark that we use the groups $\E_2(\Gamma_n(\Z))$ and not $\SL_+(\Gamma_n(\Z))$.

With a view to \Cref{sectiongroupring}, we are especially interested in the more concrete group $\SL_2(\LL)$, the special linear group of Lipschitz quaternions.
Unfortunately, in contrast to the cases $n=1$ and $n=2$, one cannot directly deduce an isomorphism between $\SL_+(\Gamma_4(\Z)) = \E_2(\Gamma_4(\Z))$  and $\SL_2(\LL)$. On a higher level, we indeed have that $\SL_+(\Gamma_4(\R)) \cong \SL_2\left(\qa{-1}{-1}{\R}\right)$ (see \Cref{isocliffquat}) but this only implies that $\SL_+(\Gamma_4(\Z))$ and $\SL_2(\LL)$ are commensurable since they are both arithmetic subgroups of these groups.
However we can prove more precisely that up to isomorphism $\SL_+(\Gamma_4(\Z)) \leq \SL_2(\LL)$ and that the index is $4$. In order to prove this, we first need the following proposition.

\begin{proposition}\label{OL_is_GE2}
$\LL$ is a $\GE_2$-ring
\end{proposition}

\begin{proof}
The proof follows exactly the same strategy as the proof of \Cref{E2=SL+}. 
Again, for $z \in \qa{-1}{-1}{\Q}$, define $\lceil z \rceil$ to be an element $z_0$ of $\LL$ that minimizes $\vert z-z_0\vert$. 
We have to make sure that for $\sm{a & b \\ c & d} \in \GL_2(\LL)$, the situation
$\vert b^{-1}a - \lceil b^{-1}a\rceil \vert = 1$ cannot occur. 
Via the Dieudonn\'e determinant, one shows again that this would lead to a contradiction. 
\end{proof}

\begin{remark}
An alternative proof for \Cref{OL_is_GE2} could also be derived from \cite[Proposition 6]{Dennis}. 
\end{remark}

In \cite{abelianizationpaper}, it is proven that the group $\SL_2(\LL)$ has property \FA. This shows that $\SL_2(\LL)$ does not have a decomposition as a non-trivial free product with amalgamation. However we will now see that it has a subgroup of finite index which has a decomposition as an amalgamated product. \Cref{OL_is_GE2} shows that $\GE_2(\LL)$ is isomorphic to $\GL_2(\LL)$, which equals $\SL_2(\LL)$ (see \Cref{Dieudonneremark}). However the groups $\GE_2(\LL)$ and $\E_2(\LL)$ are not the same. Indeed, for example, the matrix $\sm{i & 0 \\ 0 & j}$ is in $\GE_2(\LL)$, but not in $\E_2(\LL)$. By \eqref{theorem difference GE2 and E2} the index of $\E_2(\LL)$ in $\GE_2(\LL)$ is the order of the abelianization of the unit group $\U(\LL)$. As $\U(\LL)$ is simply the quaternion group $Q_8$, the index is four. Hence we get the following corollary.

\begin{corollary}\label{finite index amalgamation in SL2L}
The group $\SL_2(\LL)$ has a subgroup of index four isomorphic to a non-trivial amalgamated product.
This subgroup is $\E_2(\LL) \cong \SL_+(\Gamma_4(\Z)) = \E_2(\Gamma_4(\Z))$.
\end{corollary}

This result agrees with \cite[Lemma 5.3]{clifford}. There it is shown that $\SL_+(\Gamma_4(\Z))$ is of index four in a group, denoted $\SL_+(\tilde{\Gamma_4}(\Z))$, which is isomorphic to $\SL_2(\LL)$.

\section{A dichotomy for the unit group of an integral group ring}\label{sectiongroupring}

In the remainder of the paper we will make a step towards an answer to the virtual structure problem in case $\mathcal{G}$ is the class of groups that have decompositions as non-trivial amalgamated products. We will make use of decompositions in amalgamated products obtained in the previus section to get information about the structure of the unit group $\U (\Z G)$ of the integral group ring $\Z G$ of a finite group $G$. For this we consider actions of $\U( \Z G)$ on simplicial trees. Therefore we start with shortly reviewing property \FA. We recall some important results from \cite{abelianizationpaper} where we link property \FA to so-called exceptional simple algebras. This will settle the necessary background to prove  a dichotomy for $\U( \Z G)$, with $G$ a solvable cut group: $\U(\Z G)$ either has property \T (see definition below) or is commensurable with a non-trivial amalgamated product. Subsequently, we investigate, in the non-\T case, concrete realisations of a subgroup of finite index in $\U(\Z G)$ with a non-trivial amalgamation.

\subsection{Background}

\begin{definition}
A connected, undirected graph $T$ is said to be a \emph{simplicial tree} (or simply a \emph{tree}) if it contains no cycle graph as a subgraph. A group $\Gamma$ is said to have \emph{property \FA} if every action, without inversion, on a tree has a fixed point. 
\end{definition}

\begin{definition}
A countable discrete group $\Gamma$ is said to have \emph{Kazhdan's property \T} if and only if every affine isometric action of $\Gamma$ on a real Hilbert space has a fixed point.
\end{definition}

In this section we consider the case that $\Gamma= \U(\Z G)$, with $G$ a finite group. Recall that $\Z G$ is a $\Z$-order in the rational group algebra $\Q G$, a semi-simple algebra. By the celebrated Wedderburn-Artin Theorem, $\Q G$ is the direct product of simple algebras, i.e. $\Q G = \prod_{i=1}^m \Ma_{n_i}(D_i)$, where $m$ and $n_i$ are positive integers and each $D_i$ is a finite dimensional division algebra over $\Q$. Each simple component $\Ma_{n_i}(D_i)$ corresponds to a unique primitive central idempotent $e_i \in \Q G$ such that $\Q Ge_i \cong \Ma_{n_i}(D_i)$ (see \cite[Section 2.6]{PolSeh} for details). In \cite[Theorem 7.1]{abelianizationpaper}, it is shown that, for $\U (\Z G)$ property \T only depends on the form of the simple components $\Ma_{n_i}(D_i)$ of $\Q G$. Note that in fact the previous is shown for property \HFA (a group has property \HFA if every subgroup of finite index has property \FA). However in \cite[Corollary 7.4]{abelianizationpaper}, it is shown that for the unit group $\U(\Z G)$ of a finite group $G$, property \T and \HFA are equivalent. One of the main obstructions lies exactly in the following type of simple components.

\begin{definition}\label{definitie exceptional component}
Let $D$ be a finite dimensional division $\Q$-algebra. The algebra $\Ma_n(D)$ is called \emph{exceptional} if it is of one of the following types:
\begin{enumerate}
\item[(I)] a non-commutative division algebra other than a totally definite quaternion algebra over a number field,
\item[(II.a)] $\Ma_2(\Q(\sqrt{-d}))$ with $d \geq 0$,
\item[(II.b)] $\Ma_2\left(\qa{a}{b}{\Q}\right)$ with $a,b < 0$.
\end{enumerate}
\end{definition}

Note that this is equivalent with the definition given in \cite[Definition 6.6]{abelianizationpaper}.

The $2\times2$-matrix algebras $\Ma_2(D)$ of type (II.a) and (II.b) are exactly those where $D$ contains an order $\O$ with finite unit group (or equivalently, where any order in $D$ has finite unit group) \cite[Theorem~2.8]{abelianizationpaper}. 
Another condition for the unit group $\U(\Z G)$ to have property \T is that the group $G$ is cut. 

\begin{definition}
A finite group $G$ is said to be cut (central units trivial) if $\U( \mathcal{Z}( \Z G))$ is finite. 
\end{definition}

Indeed, in \cite[Corollary~6.3]{abelianizationpaper} we proved that a necessary condition for $\U(\Z G)$ to have finite abelianization is that the group $G$ is cut. It is well-known that property \T implies property \FA and hence also finite abelianization.

We now state the following version of \cite[Theorem~7.1]{abelianizationpaper} that encapsulates the information relevant here. Note that the original theorem is stated for property \HFA (which by \cite[Corollary 7.4]{abelianizationpaper} is equivalent to property \T) and also contains a group theoretical characterisation of property \HFA, in terms of forbidden quotients of $G$. 

\begin{theorem}[\cite{abelianizationpaper}]\label{iff HFA}
Let $G$ be a finite group. The following properties are equivalent
\begin{enumerate}
\item The group $\U(\Z G)$ has property \T.
\item $G$ is cut and $\mathbb{Q}G$ has no exceptional components.
\end{enumerate}
\end{theorem}

\begin{remark}
There seems to be an interesting link between $\U(\Z G)$ having property \T and and a subgroup generated by certain generic units in having finite index in $\U(\Z G)$. In \cite[Theorem 3.3]{JesLea2}, the authors show that if $\Q G$ does not have exceptional components, then the group generated by the generalized bicyclic units and the Bass units (see \cite[Chapter 11]{EricAngel1} for more details) is of finite index in $\U (\Z G)$. Thus, by the previous theorem, if $\U(\Z G)$ does have property \T, then the latter subgroup is of finite index in $\U(\Z G)$. Moreover, often, if $\Q G$ has exceptional components, then the subgroup generated by the generalized bicyclic units and the Bass units does not have finite index in $\U(\Z G)$. 
\end{remark}

We are interested to know what structure $\U(\Z G)$ still has when it does not have property \T. Due to \Cref{iff HFA} one might expect that the answer will depend on the precise exceptional components of $\Q G$. Therefore we summarize in the following theorem several known results that show that the possible simple algebras arising as an exceptional component of $\Q G$ are of a very limited shape, especially if $G$ is a cut group. 

\begin{theorem}\label{possible exceptional components}
Let $G$ be a finite cut group and $e$ a primitive central idempotent of $\Q G$ such that $\Q G e$ is exceptional. Then we have the following.
\begin{enumerate}
\item If $\Q Ge$ is of type $(II.a)$, then $d \in \{ 0, -1, -2 , -3 \}$.
\item If $\Q Ge$ is of type $(II.b)$, then $(a,b) \in \{ (-1,-1),( -1, -3), (-2,-5) \}$.
\item If $G$ is cut  and $\Q Ge \cong \Ma_2\left(\qa{-1}{-3}{\Q}\right)$ or $\Q Ge \cong \Ma_2(\Q(\sqrt{-2}))$ then there exists another primitive central idempotent $e'$ such that $\Q G e' \cong \Ma_2(\Q)$ or $\Q Ge' \cong \Ma_2(\Q(i))$.
 \item There exists a primitive central idempotent $e$ of $\mathbb{Q}G$ such that $\mathbb{Q}Ge \cong \Ma_2\left(\qa{-2}{-5}{\Q}\right)$ if and only if $G$ maps onto $G_{240,90}$.
\item If $G$ is solvable and cut, then $\Q Ge \ncong \Ma_2\left(\qa{-2}{-5}{\Q}\right)$.
\item If $G$ is cut, then $\Q G e$ cannot be of type $(I)$. 
\end{enumerate}
\end{theorem}

Item (1) and (2) are proven in \cite[Theorem 3.1]{EKVG} and  \cite[Theorem 3.5]{EKVG}, respectively. The three other items follow from \cite[Propositions 6.14 and 6.10]{abelianizationpaper}.

\subsection{The Dichotomy}
In the following theorem we will be interested in when the group $E_2(\O)$ is a non-trivial amalgamated product, for $\O$ some order in $\Q(\sqrt{-d})$ and $d \geq 0$ a square-free integer. Recall that we denote the maximal order of $\Q(\sqrt{-d})$ by $\mathcal{I}_{d}$. It is shown in \cite[Section 6.5, Exercise 5]{Serre} or in \cite[Theorem 5.1]{abelianizationpaper} that $\SL_2(\mathcal{I}_3)$ has property \FA. This property is in contradiction with being an amalgamated product. This is however not the case, as we will show in the proof of the following result, if one considers the order $\Z\left[\sqrt{-3}\right]$ which is a suborder of the ring of integers $\mathcal{I}_3$.
Set $B'_2(R) = \E_2(R) \cap B_2(R)$, for any ring $R$, where $\B_2(R)$ denotes the Borel subgroup of $\GL_2(R)$, i.e. the group of invertible upper triangular $2\times 2$-matrices.

\begin{theorem}\label{Zsqrt3amalgam}
Let $\mathcal{O}$ be an order in $\mathbb{Q}(\sqrt{-d})$ with $d \geq 0$ square-free. Then  $\E_2(\mathcal{O})$ has a decomposition as a non-trivial amalgamated product if and only if $\O \neq \mathcal{I}_3$. 
Moreover, if $\mathcal{O}$ is different from $\mathbb{Z}[\sqrt{-3}]$ and also from $\mathcal{I}_d$ with $d \in \lbrace 1,2,3,7,11 \rbrace$, then 
\begin{equation*}
\E_2(\mathcal{O}) \cong \E_2(\mathbb{Z}) \ast_{B'_2(\mathbb{Z})} B'_2(\O).
\end{equation*}
\end{theorem}
\begin{proof}
All the orders in our context come along with a norm $|\cdot|$. One calls $\O$ discretely normed if there exists no elements $x\in \O$ with $1 < |x| < 2$. If $\O$ is discretely normed, then, by \Cref{clifford cohn}, $\O$ is universal for $\GE_2$ and consequently \cite[Lemma 2.1.]{Menal}  yields that
\begin{equation}\label{E2O amalgam}
\E_2(\mathcal{O}) \cong \E_2(S) \ast_{B'_2(S)} B'_2(\O),
\end{equation}
where $S$ is the subring of $\O$ generated by $\U (\O)$.
Now if $\O$ is an order in $\Q$ or $\Q(\sqrt{-d})$ with $d>0$ a square-free integer, then a direct computation shows that $\O$ is discretely normed exactly when $\O$ is either $\Z$ or not a $\GE_2$-ring (i.e. $\O$ is not $\Z[\sqrt{-3}]$ or $\mathcal{I}_d$ with $d=1,2,3,7,11$ \cite[Theorem 3]{Dennis}).
For these rings $\O$, it is also not difficult to see that $\U(\O) \cong \{\pm 1 \}$ and hence $S = \Z$.
This proves the second part of the statement.

To prove the first part, we will use \eqref{E2O amalgam}. The decomposition \eqref{E2O amalgam} is non-trivial if $\O$ does not have a $\Z$-basis consisting of units. This is indeed the case for all orders satisfying \eqref{E2O amalgam} except for $\O = \Z$. Hence for all these orders different from $\Z$,  $\Z[\sqrt{-3}]$ and $\mathcal{I}_d$ for $d \in \lbrace 1,2,3,7,11 \rbrace$, $\E_2(\O)$ has a decomposition as a non-trivial amalgamated product. Note that the orders $\Z$,  $\Z[\sqrt{-3}]$ and $\mathcal{I}_d$ for $d \in \lbrace 1,2,3,7,11 \rbrace$ are all $\GE_2$-rings. Moreover, one can show that for $R$ such a ring, $\E_2(R) = \SL_2(R)$ (using \eqref{theorem difference GE2 and E2}).
If $R=\Z$ or $R=\mathcal{I}_d$ for $d=2,7,11$, then an amalgamation of $\SL_2(R)$ was obtained in \cite{Hatcher} and for $d=1$, \Cref{SL_2L_2isamalgam} also provides an amalgamation. To prove that $\E_2(\Z[\sqrt{-3}])$ has an amalgamated decomposition, we will take a completely different route.

Consider the the additive map $\varphi: \Z[\sqrt{-3}] \rightarrow \mathcal{I}_{11}$ defined by $\varphi (a +b \sqrt{-3}) = a + b \left( \frac{1+ \sqrt{-11}}{2} \right)$. This induces an epimorphism $\overline{\varphi}$ from $\E_2(\Z[\sqrt{-3}])$ to  $\E_2( \mathcal{I}_{11})$ by defining $\overline{\varphi}(E(x)) = E(\varphi(x))$. To check that $\overline{\varphi}$ is well-defined it is enough to prove that the defining relations from \Cref{clifford cohn} are preserved, but this is a straightforward calculation.
As $\E_2(\mathcal{I}_{11})$ is itself a non-trivial amalgamated product, by \Cref{amalgamconserved},  $\E_2(\Z[\sqrt{-3}])$ is also an amalgamated product.
\end{proof}

\begin{remark}
Note that applying similar methods as in \Cref{finite presentation clifford} and \ref{mainsection}, the exact decomposition as an amalgamated product of $\PSL_2(\Z(\sqrt{-3}])$ can be obtained.
The group $\PSL_2(\Z[\sqrt{-3}])$ has the following decomposition as a non-trivial amalgamated product:
$$\PSL_2(\Z[\sqrt{-3}]) = G_1 \ast_H G_2,$$
where 
$$G_1 = \langle a,m,u \mid a^2=m^2=(am)^3=1, m=u^{-1}au\rangle$$
and
$$G_2 = \langle u,s,v \mid s^3=v^3=(sv^{-1})^3=1, v=u^{-1}su\rangle$$
and $H = \langle u, am \rangle \cong \langle u, sv^{-1} \rangle$.
For more details, we refer to \cite{Doryanthesis}.
\end{remark}

With these results at hand, we are ready to prove the following dichotomy.

\begin{theorem}\label{dichotomy HFA-Amalgamated}
Let $G$ be a finite cut group, which is solvable or $5 \nmid \vert G \vert$. Then, exactly one of the following properties holds:
\begin{enumerate}
\item $\mathcal{U}(\mathbb{Z}G)$ has property \T,
\item $\mathcal{U}(\mathbb{Z}G)$ is commensurable with a non-trivial amalgamated product.
\end{enumerate}
\end{theorem}
\begin{proof}
Let $\Q G = \prod_{i \in I} \Ma_{n_i}(D_i)$ be the Wedderburn-Artin decomposition of $\Q G$, with $I$ some index set. Split $I = I_1 \cup I_2$  in such a way that $I_1$ contains all the indices $i$ such that $\Ma_{n_i}(D_i)$ is exceptional and $I_2$ contains those indices that correspond to non-exceptional components.

Since \T is a property of a commensurability class, it suffices to prove that if (1) is false, (2) holds. So suppose that $G$ is a solvable cut group and $\mathcal{U}(\mathbb{Z}G)$ does not have property \T.  As $G$ is cut, \Cref{iff HFA} implies that $I_1 \neq \emptyset$. So, by \Cref{possible exceptional components}, there exists $i_0 \in I_1$ such that $\Ma_{n_{i_0}}(D_{i_0}) = \Ma_2(D_{i_0})$ is an exceptional component of type (II.a) or (II.b). Furthermore, by \Cref{possible exceptional components}, we may assume, without loss of generality that $D_{i_0} \ncong \qa{-1}{-3}{\Q}$ and $D_{i_0} \ncong \Q(\sqrt{-2})$. If $G$ is solvable, then by \Cref{possible exceptional components}, $D_{i_0} \ncong \qa{-2}{-5}{\Q}$. If $5 \nmid \vert G \vert$, then, as explained in the proof of \cite[Theorem 3.5]{EKVG}, if $D_0$ ramifies at the prime $p$, then $p$ also divides $|G|$. Since $\qa{-2}{-5}{\Q}$ ramifies at $5$, it cannot appear as a simple component.   Hence the possibilities for $D_{i_0}$ are $\Q$, $\Q(i)$, $\Q(\sqrt{-3})$ or the classical quaternion algebra $\qa{-1}{-1}{\Q}$. In all these cases $D_{i_0}$ has a unique, up to conjugation, maximal order (for the commutative case this is just the ring of integers; for the quaternion algebra we refer to \cite{CeChLez}) which we denote by $\OO_{i_0}$. In view of \cite[(21.6), page~189]{Reiner} this yields that $\Ma_2(\OO_{i_0})$ is also up to conjugation the unique maximal order in $\Ma_2(D_{i_0})$. Since the reduced norm has image in a maximal order in $\ZZ(D_{i_0})$, which has finite unit group in all relevant cases, $\SL_2(\O_{i_0})$ has finite index in $\GL_2(\O_{i_0})$. Thus by \Cref{SL_2L_2isamalgam}, \Cref{Zsqrt3amalgam} and \Cref{finite index amalgamation in SL2L}, respectively, $\GL_2(\OO_{i_0})$ contains a subgroup of finite index, say $H$, which has a non-trivial decomposition as amalgamated product. Finally, \Cref{amalgamconserved} yields that $H \times \prod_{i\in I\setminus \lbrace i_0 \rbrace} \GL_{n_i}(\OO_i)$ inherits from $H$ a non-trivial decomposition as amalgamated product. As orders have commensurable unit group \cite[Lemma~4.6.9]{EricAngel1}, the latter finishes the proof.
\end{proof}

\begin{remark}\label{H5 excluded}
The property of $G$ being solvable or $5 \nmid |G|$ only matters to make sure that $D_{i_0} \ncong \qa{-2}{-5}{\Q}$ (see \Cref{possible exceptional components}). Therefore the same theorem holds for every finite group $G$ for which $\Q G$ does not have $\Ma_2\left(\qa{-2}{-5}{\Q}\right)$ as a simple component. 

This obstruction is due to the fact that we do not know whether there is a subgroup of finite index in $\GL_2(\O) = \GE_2(\O)$, for $\O$ the unique (up to conjugation) maximal order in $\qa{-2}{-5}{\Q}$,  that has a decomposition as a non-trivial amalgamated product. Hence the condition of $G$ being solvable in the previous theorem.
\end{remark}

\begin{question}\label{que:O5}
Does $\GL_2(\O)$, for $\O$ the unique maximal order in $\qa{-2}{-5}{\Q}$ have a decomposition as amalgamated product? Is there a subgroup of finite index in $\GL_2(\O)$ that has a decomposition as an amalgamated product?
\end{question}

Recall that we defined the concepts of reduced norm and $\SL_1$ for a subring of a central simple algebra in \Cref{matquat}. In what follows, we will frequently need the notion $\SL_1(R)$ for $R$ a subring in a semisimple $\Q$-algebra $A$. Let $A = \prod \Ma_{n_i}(D_i)$ be the Wedderburn-Artin decomposition of $A$ and $h_i$ the projections onto the $i$-th component. Then $$\SL_1(R):=\{a \in R\ \mid \ \forall\ i \colon \operatorname{RNr}_{\Ma_{n_i}(D_i)/\ZZ(D_i)}(h_i( a)) = 1\}.$$ 

The next natural question is whether it is possible, in case $\U (\Z G)$ is non-\T, to find a concrete realisation of this amalgamated product inside $\U (\Z G)$, i.e.\ can we describe a subgroup $H$ of finite index in $\U (\Z G)$ which is in a non-trivial way an amalgamated product.  We expect that $H= \SL_1 (\Z G)$ is always such a realisation.

\begin{question}\label{conjecture concrete amalgam}
Let $G$ be a finite group. Are the following properties equivalent?
\begin{enumerate}
\item $\U(\mathbb{Z}G)$ is, up to commensurability, a non-trivial amalgamated product.
\item  $\SL_1(\mathbb{Z}G)$ is a non-trivial amalgamated product
\end{enumerate}
\end{question}
If $G$ is a finite cut group, then $\SL_1(\Z G)$ is of finite index in $\U (\Z G)$ \cite[Proposition~5.5.1]{EricAngel1} and hence if $\Q G$ has no exceptional components, then $\SL_1( \Z G)$ has property \T by \Cref{dichotomy HFA-Amalgamated}. Therefore, from now on, we may assume the existence of an exceptional component $\Ma_2(D_{exc})$ whose form is given by \Cref{possible exceptional components}. We want to determine a precise subgroup of finite index in $\SL_2(\O_{exc})$, where $\O_{exc}$ is an order in $D_{exc}$, with a non-trivial amalgamated product and retract it to a subgroup of finite index in $\SL_1(\Z G)$. We start by considering the case  $\Ma_2(D_{exc}) = \Ma_2(\Q)$ which is a very generic case, as can be seen by inspecting \cite[Appendix A]{abelianizationpaper}.

\begin{proposition}\label{amalgamation due to M_2(Q)}
Let $A$ be a finite dimensional semisimple $\mathbb{Q}$-algebra and $\OO$ an order in $A$. Suppose $\Ma_2(\mathbb{Q})$ is a Wedderburn-Artin component of $A$ and let $H$ be a subgroup of finite index in  $\SL_1(\O)$. Then $H$ has a non-trivial amalgamated decomposition.
\end{proposition}
\begin{proof}
Let $A = \prod_{i \in I} \Ma_{n_i}(D_i)$ be the Wedderburn-Artin decomposition and $\{ e_i \mid i \in I \}$ the associated system of primitive central idempotents. For the remainder of the proof, fix $i_0 \in I$ such that $Ae_{i_0} \cong \Ma_2(\Q)$. Recall that $\O e_i$ is an order in $A e_i$ for all $i$. Further, it is well known that $\Ma_2(\Z)$ is up to conjugation the unique maximal order of $\Ma_2(\Q)$. 

Assume that $H$ is a subgroup of finite index in $\SL_1(\O)$. By definition, $\SL_1(\O)$ can be viewed as a subgroup of $\Gamma = \SL_2(\Z) \times \prod_{i \in I \setminus \{ i_0 \} } \SL_1(\O e_i)$. By \cite[Lemma 4.6.6, 4.6.9 and Proposition 5.5.1]{EricAngel1} one readily proves that $\SL_1(\O)$, hence also $H$, has finite index in $\Gamma$. By a classical result (or the results in \Cref{mainsection}), $\SL_2(\Z) \cong C_4 \ast_{C_2} C_6$ with $C_2 = \langle c \rangle$ a central and hence normal subgroup. By \Cref{amalgamconserved} we get that $\Gamma = B \ast_U D$, with $B \cong C_4 \times \prod_{i \in I \setminus \{ i_0 \} } \SL_1( \O e_i)$, $D \cong C_6 \times \prod_{i \in I \setminus \{ i_0 \} } \SL_1( \O e_i)$ and $U \cong C_2 \times \prod_{i \in I \setminus \{ i_0 \} } \SL_1( \O e_i)$. In particular, $U$ is normal in $B$ and $D$. Now consider $\Gamma/U \cong B/U \ast D/U$ and its subgroup of finite index $H / (H\cap U)$.
It is easy to show that $B/U$ and $D/U$ are of infinite index in $\Gamma/U$ and by the Kurosh subgroup theorem, $H/ (H\cap U)$ is a non-trivial free product. Hence, by \Cref{amalgamconserved}, $H$ has a non-trivial amalgamated decomposition.
\end{proof}

We do not know if the previous proposition stays true when replacing $\SL_1(\O)$ by $\U(\O)$. Nevertheless, the following presentation of $\GL_2(\Z)$ is well-known (for instance see \cite{CoxMos}).
$$\GL_2(\Z) = \langle x,y,z \mid x^2 = y^2 = z^2 = 1, (xy)^3 =(xz)^2, (xz)^4 = 1 \rangle.$$
By \Cref{disjoint generators amalgam}, $\GL_2(\Z)$ has a non-trivial decomposition as an amalgamated product and hence it makes sense asking the question if $\SL_1(\O)$ could be replaced more generally by $\U(\O)$.

\begin{corollary}\label{subgroups of SL_2(Z) are amalgam}
Every finite index subgroup of $\SL_2(\Z)$ is a non-trivial amalgamated product.
\end{corollary}
\begin{proof}
This follows from \Cref{amalgamation due to M_2(Q)} applied to $A = \Ma_2(\Q)$ and $\O = \Ma_2(\Z)$.
In that case $\SL_1(\O) = \SL_2(\Z)$.
\end{proof}

In case $A = \Q G$ the condition in \Cref{amalgamation due to M_2(Q)} can be reformulated in terms of $G$.

\begin{corollary} \label{amalgam voor QG als M_2(Q) component}
Let $G$ be a finite group having $D_8$ or $S_3$ as an epimorphic image. If $H$ is a subgroup of finite index in $\SL_1(\Z G)$, then $H$ is a non-trivial amalgamated product. Moreover, $\U(\Z G)$ is virtually an amalgamated product. 
\end{corollary}
\begin{proof}
Let $N$ be a normal subgroup of $G$ such that $G/N \cong D_8$ or $S_3$. It is well known (see for example \cite[Proposition 3.6.7]{PolSeh}) that $\Q G = \Q G e_N \times \Q G (1-  e_N)$, where $e_N = \frac{1}{|N|} \sum_{n \in N} n$ the (central) idempotent associated to $N$, and $\Q G e_N \cong \Q G/N$. Now, since $D_8$ and $S_3$ have an irreducible $\Q$-representation of degree $2$, $\Q G/N$ has $\Ma_2(\Q)$ as a simple component. Hence, by the above, so does $\Q G$. The first part of the result now follows from \Cref{amalgamation due to M_2(Q)}.

For the second part, note that $\langle \SL_1(\U (\Z G)), \mathcal{Z}(\U(\Z G)) \rangle$ is always of finite index in $\U(\Z G)$ (see for example \cite[Proposition 5.5.1]{EricAngel1}). As $\SL_1(\Z G)$ is a a non-trivial amalgamated product and $\mathcal{Z}(\U(\Z G))$ is central in $\langle \SL_1(\U (\Z G)), \mathcal{Z}(\U(\Z G)) \rangle$, the whole group $\langle \SL_1(\U (\Z G)), \mathcal{Z}(\U(\Z G)) \rangle$ is a non-trivial amalgamated product. 
\end{proof}

\begin{remark}\label{remark over wnr M-2(Q) component}
As follows from the proof, $\Ma_2(\Q)$ is a simple component of $\Q G$ if and only if $G$ contains a normal subgroup $N$ such that $G/N$ faithfully embeds in $\GL_2(\Q)$ and it is non-abelian. Since the only finite non-abelian subgroups of $\GL_2(\Q)$ are the dihedral groups $D_{2n}$ of order $2n$ for $n \in \{3, 4, 6\}$, it follows that $\Q G$ has a $\Ma_2(\Q)$ as component if and only if $D_8$ or $D_6 \cong S_3$ is an epimorphic image of $G$.
\end{remark}

Now we consider the other components in \Cref{possible exceptional components}. Unfortunately, such a strong result as in \Cref{amalgamation due to M_2(Q)} is impossible to obtain for $D_{exc} = \Q(\sqrt{-3})$ or $\qa{-1}{-1}{\Q}$ since in \cite{abelianizationpaper} it is shown that $\SL_2(\I_3)$ and $\SL_2(\O_2)$ are \FA. Here $\O_2$ denotes the maximal order in $\qa{-1}{-1}{\Q}$ having $\{1, i,j,\frac{1+i+j+k}{2} \}$ as a basis (this order is also called Hurrwitz order).  
In contrast, when $\D_{exc} = \Q(i)$, then $\SL_2(\Z[i]) \cong \SL_+(\Gamma_2(\Z))$ is an amalgamated product.
As such we know that every subgroup of finite index does not have \FA, but we're unsure whether it has an amalgamated decomposition or an infinite abelianization.

\begin{question}\label{question SL_2}
Is every finite index subgroup of $\SL_2(\Z [i])$ a non-trivial amalgamated product?
\end{question}

For the last proposition, we recall the definition of an HNN extension. Let $\Gamma$ be a group with presentation $\langle S \mid R \rangle$, $H_1$ and $H_2$ be two isomorphic subgroups of $\Gamma$ and  $\theta: H_1 \rightarrow H_2$ an isomorphism. Let $t \not \in \Gamma$ be a new element and $\langle t \rangle$ a cyclic group of infinite order. The \textit{HNN extension} of $\Gamma$ relative to $H_1$, $H_2$ and $\theta$ is the group 
$$\langle S, t \mid R, tgt^{-1}=\theta(g), g\in H_1 \rangle.$$

By classical Bass-Serre theory, a finitely generated group is an HNN extension if and only if it has infinite abelianization, i.e.\ it maps onto $\Z$. Thus a finitely generated group has property \FA exactly when it is neither an HNN extension nor a non-trivial amalgam.

The following proposition gives a concrete subgroup of $\SL_1(\Z G)$ that has a decomposition as an amalgamated product or an HNN extension. Note that the indices of the groups $G_{32, 50}$, $G_{96,202}$ and $G_{384, 618}$ appearing in the proposition indicate their \textsc{SmallGroup ID}s in the Small Group library of \textsf{GAP} \cite{GAP}. For a presentation of the groups, we refer to \cite[Appendix B]{abelianizationpaper}.

\begin{proposition}\label{prop:am_HNN}
Let $G$ be a finite cut group such that $G$ does not map onto $$\SL(2,3),\quad C_3 \times Q_8,\quad G_{32, 50},\quad G_{96,202}, \quad G_{240, 90} \quad \mbox{ or }\quad G_{384, 618}.$$ Suppose that $\U(\Z G)$ does not have property \T. Then any subgroup of finite index in $\SL_1(\Z G)$ has a non-trivial decomposition as amalgamated product or is an HNN extension.
\end{proposition}

\begin{proof}
Let $\Q G = \prod_{i \in I} \Ma_{n_i}(D_i)$ be the decomposition into simple algebras and let $\{ e_i \mid i \in I \}$ be the associated system of primitive central idempotents. By \Cref{iff HFA,possible exceptional components} there exists $i_0 \in I$ such that $\Q G e_{i_0} = \Ma_2(D_{i_0})$ is an exceptional component of type (II.a) or (II.b). Again by  \Cref{possible exceptional components}, we may assume that $D_{i_0} \in \{ \Q, \Q(i), \Q(\sqrt{-3}), \qa{-1}{-1}{\Q}\}$ (since $G$ does not map onto $G_{240, 90}$, by assumption,  $D_{i_0} \not\cong \qa{-2}{-5}{\Q}$). Now from \cite[Appendix~A]{abelianizationpaper} we can deduce:
\begin{enumerate}
\item $G$ maps onto $\SL(2,3)$ or onto $C_3 \times Q_8$ if and only if $\Ma_2(\Q(\sqrt{-3}))$ is a component of $\Q G$, but $\Ma_2(\Q)$ is not.
\item $G$ maps onto $G_{32, 50}, G_{96,202}$ or onto $G_{384, 618}$ if and only if $\Ma_2\left(\qa{-1}{-1}{\Q}\right)$ is a component, but not $\Ma_2(\Q)$.
\end{enumerate}
Thus we may further assume that $D_{i_0} = \Q$ or $D_{i_0} = \Q(i)$. Let $\O_{i_0}$ be a maximal order in $D_{i_0}$ (which is unique up to conjugation).  By \Cref{SL_2L_2isamalgam}, $\SL_2(\O_{i_0})$ is non-trivially an amalgamated product. Set $\Gamma := \SL_2(\O_{i_0}) \times \prod_{i \in I \setminus \{ i_0 \} } \SL_1(\Z G e_i)$. This is again an amalgamated product and hence $\Gamma$ does not have property \FA. Thus every finite index subgroup of $\Gamma$ does not have property \FA. As explained before the proposition, this implies that every subgroup of finite index in $\Gamma$ is an HNN extension or a non-trivial amalgamated product. As $\SL_1(\Z G)$ has finite index in $\Gamma$, the result follows. \end{proof}

The group $G_{240,90}$ in the previous proposition is excluded as epimorphic image of $G$ just to prevent the occurrence of $\Ma_2\left(\qa{-2}{-5}{\Q}\right)$ as a Wedderburn-Artin component of $\Q G$. If \Cref{que:O5} has a positive answer, excluding this component and hence also the epimorphic image $G_{240, 90}$ would no longer be necessary. Unfortunately, this component cannot be handled in the same way as $\Ma_2\left(\qa{-1}{-3}{\Q}\right)$, since $\Ma_2\left(\qa{-2}{-5}{\Q}\right)$ is the unique type of exceptional components of type (II) in the cut group $G_{240, 90}$.

\vspace{1cm}

\noindent \textbf{Acknowledgement.} We would like to thank Yves de Cornulier for some helpful insights on amalgamated products. 

\appendix

\chapter{A non-exotic example of a quaternion order that is not universal for $\GE_2$}\label{appendix}

Let $\O$ be an order in a division algebra, equipped with a norm function $|\cdot|$. Recall that $\O$ is said to be discretely normed if there exists no elements $x\in \O$ with $1 < |x| < 2$. It is well-known that the maximal orders in imaginary quadratic extensions of $\Q$ are either discretely normed or Euclidean. If they are discretely normed, then by \Cref{clifford cohn}, they are universal for $\GE_2$, i.e. the universal relations form a complete set of relations for $\GE_2(\O)$. In \cite{Cohn2}, it is shown that among the Euclidean maximal orders, only the two orders having a basis consisting of units, i.e. $\mathcal{I}_{1} \subseteq \Q(\sqrt{-1})$ and $\mathcal{I}_3 \subseteq \Q(\sqrt{-1})$, are universal for $\GE_2$. In this Appendix, we show that a similar result holds for the Euclidean orders in totally definite rational quaternion algebras.

By \cite{Fitz} the only totally definite rational quaternion algebras having a right Euclidean order are $\qa{-1}{-1}{\Q}$,$\qa{-1}{-3}{\Q}$ and $\qa{-2}{-5}{\Q}$. Note that orders that are (right norm) Euclidean are maximal \cite[Proposition 2.8]{CeChLez}. Furthermore a quaternion algebra having a right norm Euclidean order has class number one \cite[Proposition 2.9]{CeChLez} and thus also type number one meaning that there is only one conjugacy class of maximal orders. In \cite{Fitz} also a specific representative of that unique conjugacy class, denoted $\O_2$, $\O_3$ and $\O_5$ respectively, is constructed. In the table below, we state specific $\mathbb{Z}$-bases $\{b_1, b_2, b_3, b_4\}$ of these orders (which also can be found in \cite[Proposition~12.3.2]{EricAngel1}). 
\newcolumntype{C}{>{\centering\arraybackslash}p{2.5cm}}
\newcolumntype{B}{>{\centering\arraybackslash}p{2cm}}
\begin{equation*}\label{eq:basis_of_quat_orders} 
\begin{tabular}{|*{2}{B|}*{3}{C|}}
& $b_1$ & $b_2$ & $b_3$ & $b_4$ \\ \hline
  $\O_2$ & $1$ & $i$ & $j$ & $\omega_2 = \frac{1 + i + j + k}{2}$ \\
  $\O_3$ & $1$ & $i$ & $\omega_3 = \frac{1+j}{2}$ & $\frac{i+k}{2}$  \\
  $\O_5$ & $1$ & $ \frac{1+i+j}{2}$ & $\omega_5 = \frac{2+i-k}{4}$ & $\frac{2+3i+k}{4}$ \\ \hline
\end{tabular} \end{equation*}

The quaternion algebra $\qa{-1}{-1}{\Q}$ also contains the order of \emph{Lipschitz quaternions} $\LL$ consisting of all integral linear combinations of the basis elements $1, i, j, k$. Although $\LL$ is not Euclidean, \Cref{OL_is_GE2} shows that it also is a $\GE_2$-ring, like $\O_2$, $\O_3$ and $\O_5$. The basis of the orders $\LL$, $\O_2$ and $\O_3$ consist of units. This allows to apply a similar reasoning as in the proof of \Cref{universalcomplete} and prove that these three orders are universal. 

 On the contrary, the order $\O_5$ does not have a basis of units and hence the proof of \Cref{universalcomplete} is not applicable. Even stronger, by using an elegant result from \cite{Cohn1} involving $\U$-homomorphisms, we show that the universal relations do not form a complete set of relations for $\GE_2(\O_5)$. A map $f: R \rightarrow S$ between two unital rings $R$ and $S$ is called a $\U$-homomorphism if it is a morphism between the additive groups of the rings mapping $1_R$ to $1_S$ and moreover if for any $a \in R$ and units $\alpha, \beta$ of $R$ the following holds:
 $$f(\alpha a \beta) = f(\alpha)f(a)f(\beta). $$
Let $R$ be a ring, such that the universal relations form a complete set of relations for $\GE_2$, $S$ any ring and let $f \colon R \to S$ be any $\U$-homomorphism. Then Cohn's result  \cite[Theorem (11.2)]{Cohn1} states that $f$ induces a homomorphism $f^* \colon \GE_2(R) \to \GE_2(S)$ by the rule $E(x) \mapsto E(f(x))$, $[\alpha, \beta] \mapsto [f(\alpha), f(\beta)]$. One may use this result ad absurdum to prove that the universal relations do not form a complete set of relations for $\GE_2$ of a certain ring.

Suppose that the universal relations give a complete set of relations for $\GE_2(\O_5)$.
It is straightforward to check that $f : \O_5 \rightarrow \O_2$, defined by the $\Z$-linear expansion of \begin{center}\begin{tabular}{rcllrcl}
 $1$ & $\mapsto$ & $1',$ & & $\frac{2+i-k}{4}$ &  $\mapsto$ & $\frac{1'+i'-j'+k'}{2},$\\
 $\frac{1+i+j}{2}$ & $ \mapsto$ &  $\frac{1'+i'+j'+k'}{2}$, & & $ \frac{2+3i+k}{4}$ &  $\mapsto$ & $\frac{1'-i'+j'+k'}{2}$ 
\end{tabular}\end{center}
is a $\U$-homomorphism where we denoted the standard quaternion basis of $\qa{-1}{-1}{\Q}$ by $1', i', j', k'$.
By Cohn's result it follows that the induced map $f^*: \GE_2(\O_5) \rightarrow \GE_2(\O_2)$ is a morphism. Since $\frac{1+i+j}{2}$ is an element of norm $2$ in $\O_5$, \Cref{clifford cohn} claims that $$\left( E\left(1-\frac{1+i+j}{2}\right)E\left(\frac{1+i+j}{2}\right)\right)^2 = E(0)^2 $$

Taking $f^*$ on both sides yields $$\left( E\left(1'-\frac{1'+i'+j'+k'}{2}\right)E\left(\frac{1'+i'+j'+k'}{2}\right)\right)^2 = E(0)^2, $$ which does clearly not hold in $\O_2$. 

To sum up, let $\LL$ denote the Lipschitz quaternions in $\qa{-1}{-1}{\Q}$ and let $\O_2$, $\O_3$ and $\O_5$ denote the maximal order in $\qa{-1}{-1}{\Q}$, $\qa{-1}{-3}{\Q}$ and $\qa{-2}{-5}{\Q}$ respectively, stated in \cite[(3.20)]{abelianizationpaper}. 
By the previous and the explanation after \Cref{universalcomplete}, we have the following proposition. 

\begin{proposition*}
The universal relations form a complete set of relations for $\GE_2(\LL)$, $\GE_2(\O_2)$ and $\GE_2(\OO_3)$. This does not hold for $\GE_2(\O_5)$.
\end{proposition*}

\bibliographystyle{plain}
\bibliography{FA}

\end{document}